\documentclass[review,onefignum,onetabnum]{siamart190516}

\usepackage{lipsum}
\usepackage{amsfonts}
\usepackage{graphicx}
\usepackage{epstopdf}
\usepackage[noend]{algorithmic}
\ifpdf
  \DeclareGraphicsExtensions{.eps,.pdf,.png,.jpg}
\else
  \DeclareGraphicsExtensions{.eps}
\fi

\newsiamremark{remark}{Remark}
\newsiamremark{hypothesis}{Hypothesis}
\crefname{hypothesis}{Hypothesis}{Hypotheses}
\newsiamthm{claim}{Claim}
\newtheorem{assumption}[theorem]{Assumption}
\newtheorem{example}{Example}

\headers{Optimization for upper-$C^2$ objective with smooth constriants}{J. Wang and C. G. Petra}

\title{An optimization algorithm for nonsmooth nonconvex problems with upper-$C^2$ objective
        \thanks{Submitted to the editors.
\funding{Prepared by Lawrence Livermore National Laboratory (LLNL) under Contract DE-AC52-07NA27344.}}}

\author{Jingyi Wang\thanks{Center for Applied Scientific Computing, Lawrence Livermore National Laboratory,
Livermore, CA 
  (\email{wang125@llnl.gov, petra1@llnl.gov}).}
\and Cosmin G. Petra\footnotemark[2]
}

\usepackage{amsopn}

\ifpdf
\hypersetup{
  pdftitle={An optimization algorithm for nonsmooth nonconvex problems with upper-$C^2$ objective},
  pdfauthor={J. Wang and C. G. Petra}
}
\fi

\begin{document}
\nolinenumbers
\maketitle

\begin{abstract}
An optimization algorithm for nonsmooth nonconvex constrained optimization problems with upper-$C^2$ objective functions is proposed and analyzed. Upper-$C^2$ is a weakly concave property that exists in difference of convex (DC) functions and arises naturally in many applications, particularly certain classes of solutions to parametric optimization problems~\cite{rockafellar1998,bonnans_book}, \textit{e.g.}, recourse of stochastic programming~\cite{Shapiro_book} and projection into closed sets~\cite{rockafellar1998}. The algorithm can be viewed as a bundle method specialized for upper-$C^2$ problems and is globally convergent with bounded algorithm parameters. Compared to conventional bundle methods, the proposed method is both simpler and computationally more efficient. The algorithm handles general smooth constraints similarly to sequential quadratic programming (SQP) methods and uses a line search to ensure progress. The potential inconsistencies from the linearization of the constraints are addressed through a penalty method. The capabilities of the algorithm 
are demonstrated by solving both simple upper-$C^2$ problems and real-world optimal power flow problems used in current power grid industry practices.

\end{abstract}




\newcommand{\Rbb}{\ensuremath{\mathbb{R} }}
\section{Introduction}\label{se:intro}

In this paper, we consider the class of nonsmooth nonconvex  optimization problems in the form of 
\begin{equation} \label{eqn:opt0}
 \centering
  \begin{aligned}
	  &\underset{\substack{x}}{\text{minimize}} 
	  & & f(x)+ r(x)\\
   &\text{subject to}
	  & & c(x) = 0, \ \ \ d(x) \geq 0,\\
  \end{aligned}
\end{equation}
where the objective function $f:\Rbb^n\to\Rbb$, and constraints $c:\Rbb^n\to\Rbb^{m_c}$, $d:\Rbb^n\to\Rbb^{m_d}$ are continuously differentiable. 
The function $r:\Rbb^n\to\Rbb$ is Lipschitz continuous and  
upper-$C^2$ (see monograph~\cite{rockafellar1998} and Section~\ref{sec:prob}),
a property that is rarely exploited individually by numerical algorithms but, nevertheless, exists in many optimization problems 
as we elaborate later in this section. 
Nonsmooth optimization has been researched extensively for decades. 
Most successful methods include subgradient methods~\cite{shor1985}, bundle methods, and DC algorithms~\cite{an2018}. 
Bundle methods are widely regarded as
one of the most efficient algorithms to address discontinuous first-order derivatives~\cite{makela1992,Kiwiel1996}. 
The method develops an approximation model for the objective with the information from previous iterations, 
referred to as a bundle, and solves optimization subproblems with it~\cite{mifflin1982,kiwiel1985}. 
Many bundle algorithms rely on the quadratic coefficient in the approximation model to avoid line searches~\cite{hare2010,noll2013}. 
The solution to the subproblem is called a trial point/step, which through a rejection criterion is either taken or rejected, but included in the bundle to improve the trial point for the next iteration.
A trial point is called a serious point if it is accepted.
The approximation typically consists of piecewise-linear tangent planes at the bundle points 
and can be adjusted to maintain positive linearization error that is lacking due to the nonconvexity of the objective.
A commonly used adjustment, the so-called down-shift mechanism, is introduced in~\cite{mifflin1982} and investigated in~\cite{lemarechal1996,lemarechal1978,schramm1992,zowe1989}. Convergence analysis using this mechanism can be found in~\cite{apkarian2008,noll2009}. 
Given better local convexity properties, \textit{e.g.}, lower-$C^2$, the slope of the tangent planes can be tilted to generate locally convexified models~\cite{hare2010}. The so-called redistributed bundle methods are shown to work in practice under less ideal conditions~\cite{hare2015}.
Constrained nonsmooth nonconvex optimization adds another layer of complexity for bundle methods. Constraints that form a convex feasible set, \textit{e.g.}, affine constraints, can be maintained as-is in the subproblems and convergence analysis would stand valid~\cite{hare2015,dao2016}.
For a nonconvex $c(\cdot)$, smooth or not, bundle methods typically incorporate it into the objective through penalty or filter methods so that the approximation is constructed for the new objective. 


Outside of bundle methods,~\cite{curtis2012} proposed a sequential quadratic programming (SQP) method using gradient sampling for inequality-constrained optimization. 
The global convergence result of the algorithm shows that accumulation points are stationary points of the exact penalty function, which can be reduced to the constraint violation depending on the penalty parameter. 
A more efficient BFGS-SQP is proposed in~\cite{curtis2017} which shows promising convergence behaviors without requiring the existence of Hessian.
In \cite{xu2015}, a smoothing function of the objective that satisfies gradient consistency property and augmented Lagrangian constraint relaxation are used to solve the problem. 
Alternating direction method of
multipliers (ADMM) has also been applied to nonsmooth nonconvex problems~\cite{hong2018}. In~\cite{wang2019} the convergence analysis requires the objective to be locally prox-regular with affine constraints. 
DC functions include both lower-$C^2$ and upper-$C^2$ functions, though the latter is seldom studied separately. For an overview of DC functions and the celebrated DC algorithms (DCA), we refer to~\cite{an2018,cui2021}. In~\cite{liu2020}, the recourse function of linearly bi-parameterized two-stage problems with quadratic recourse is shown to have convexity-concavity property. The authors proposed an iterative algorithm with a quadratic convex subproblem that converges subsequently to generalized critical points. 

Upper-$C^2$ functions exist in many optimization applications and can be viewed as DC functions. 
Any finite, concave function is upper-$C^2$~\cite[Theorem~10.33]{rockafellar1998}, as are continuously differentiable functions~\cite[Proposition~13.34]{rockafellar1998}.
Moreover, a squared distance function to a closed set, which is  
the solution function of a minimization problem, is upper-$C^2$~\cite[Example~10.57]{rockafellar1998}.
The property extends to a large number of two-stage (stochastic) programming problems with recourse~\cite{Shapiro_book,Birge97Book,KallWallace}, 
whose objective includes an upper-$C^2$ recourse function of the second-stage problem.
The application that motivated us to look into such concavity-related property is 
the distributed security-constrained alternating current optimal power flow (SCACOPF) problems~\cite{ChiangPetraZavala_14_PIPSNLP,Qiu2005,petra_14_augIncomplete,petra_14_realtime,petra_21_gollnlp}, which can be formed as a two-stage recourse problem. 
In this case, the nonsmooth part of the first-stage objective $r(\cdot)$ becomes upper-$C^2$ through regularization of the second-stage problems, if it is not already so~\cite{wang2022}. 
We point out that upper-$C^2$ does not guarantee differentiability, lower-$C^1$ or lower regularity (see Section~\ref{sec:prob}).

The existing algorithms unfortunately do not fully address the challenges in solving~\eqref{eqn:opt0}.
Without the explicit form of $r(\cdot)$, potential algorithm needs to rely on information at known points to 
construct its approximation. Further, the cost of evaluating $r(\cdot)$ and its subgradient could be substantial, 
which incentivizes the algorithm to avoid line search on the objective. Bundle methods are good algorithm candidates.
However, in order to achieve convergence, existing bundle methods either rely on additional properties of $r(\cdot)$, \textit{e.g.}, lower-$C^2$~\cite{hare2010,hare2015,yang2014} or allow the approximation model coefficient to be unbounded~\cite{noll2009,noll2013,dao2015}. Neither is ideal for $r(\cdot)$ that do not enjoy such properties. 
The complex selection criteria of the bundle points and the cost of incorporating them in the 
subproblem could also be a serious drawback. DCA, which in essence uses a linear function to approximate the concave function, bear some similarities to the proposed algorithm. The existence of both equality and inequality constraints makes the exact penalty method used in some DCAs to address constraints challenging. Given the absence of the convex part of the DC function, a flexible and robust quadratic term would help guarantee convergence and control the step size.
To provide theoretically sound and computationally efficient algorithms for~\eqref{eqn:opt0}, 
we propose a simplified bundle method that capitalizes on the upper-$C^2$ objective and the continuously differentiable, nonconvex constraints. 
The proposed algorithms has been parallelized efficiently for SCACOPF problems~\cite{wang2021}.

The paper is organized as follows. In Section~\ref{sec:prob}, we describe the upper-type properties of $r(\cdot)$ and mathematical notations.
In Section~\ref{sec:alg}, our algorithm is proposed and its global convergence analysis  
is provided under assumptions drawn from Section~\ref{sec:prob}.
We discuss how the parameters of the algorithm are updated for the approximation models used in subproblems.
An algorithm to address possible inconsistency arising from the linearized constraints in the subproblems is then presented,
together with its convergence analysis.
Numerical experiments are shown in Section~\ref{sec:exp} that illustrate both the 
theoretical and practical capabilities of the proposed algorithm.

\newcommand{\norm}[1]{\left\lVert {#1} \right\rVert}
\section{Background and notations}\label{sec:prob}

In this section, we provide the concepts and notations necessary for the analysis in this paper. 
The lower regular subdifferential of a function $r:\Rbb^n\to\Rbb$ at point $\bar{x}$, 
denoted as $\hat{\partial} r(\bar{x})$, is defined by  
\begin{equation}\label{eqn:subgradient-def}
\centering
 \begin{aligned}
	\hat{\partial} r (\bar{x}) := \left\{ g\in\Rbb^n |\liminf_{ \substack{x\to \bar{x}\\x\neq \bar{x} }}\frac{r(x)-r(\bar{x})-\langle g,x-\bar{x}\rangle  }{\norm{x-\bar{x}}}\geq 0\right\},
 \end{aligned}
\end{equation}
where $\norm{\cdot}$ is the 2-norm and $\langle \cdot \rangle$ is the inner product in $\Rbb^n$.
If there exists a sequence $\{x^{\nu}\}$ such that $x^{\nu} \xrightarrow[r]{} \bar{x}$ and $g^{\nu}\in \hat{\partial}r(x^{\nu})$ with $g^{\nu}\to \bar{g}$, then 
$\bar{g}$ is a lower general subgradient of $r(\bar{x})$, where the f-attentive convergence 
$ x^{\nu} \xrightarrow[r]{} \bar{x}$ is trivial for $x^{\nu}\to \bar{x}$ when $r$ is Lipschitz. 
The lower general subdifferential at $\bar{x}$ is denoted as $\partial r(\bar{x})$.

A Lipschitz function $r$ is lower regular (or subdifferentially regular) 
if and only if $\partial r(\bar{x})=\hat{\partial} r(\bar{x})$~\cite[Corollary~8.11]{rockafellar1998}. 
Lower general subgradient is often simply called general subgradient, 
and a lower regular function is called regular. 
On the other hand, 
upper regular subdifferential~\cite{rockafellar1998,mordukhovich2004upp} is defined as
\begin{equation}\label{def:upp-subgradient}
\centering
 \begin{aligned}
	 \hat{\partial}^+ r(\bar{x}) :=& - \hat{\partial} (-r)(\bar{x})
	 =\left\{ g\in\Rbb^n |\limsup_{ \substack{x\to \bar{x}\\x\neq \bar{x} }}\frac{r(x)-r(\bar{x})-\langle g,x-\bar{x}\rangle  }{\norm{x-\bar{x}}}\leq 0\right\}.
 \end{aligned}
\end{equation}
Similarly, the upper general subdifferential is given by $\partial^+ r(\bar{x}) := -\partial (-r)(\bar{x})$.
A function $r$ is called upper regular if $-r$ is lower regular. 
Examples of upper regular functions include all continuous concave and continuously differentiable functions.

In nonsmooth nonconvex optimization literature, both Clarke subdifferential~\cite{clarke1983}, denoted as $\bar{\partial} r{(\bar{x})}$ of $r$ at $\bar{x}$, and $\partial r(\bar{x})$ have been widely adopted.
If $r$ is lower regular, given the convexity of $\hat{\partial} r(\bar{x})$, $\partial r(\bar{x}) =\bar{\partial} r{(\bar{x})}$~\cite[Theorem~8.6,~8.49,~9.61]{rockafellar1998}.
This equivalence holds for an upper regular $r$ as well (for proof see~\cite{wang2022}) and allows us to use its Clarke subgradient and upper general subgradient interchangeably.  
An important property of these subdifferential is the outer/upper-semicontinuity, necessary in establishing convergence~\cite[Proposition~6.6]{rockafellar1998}. 
In addition, for a Lipschitz $r$, $\bar{\partial} r(\bar{x})$ is locally bounded~\cite[Theorem~9.13]{rockafellar1998}. 

A more restrictive but useful assumption than regularity is lower-$C^k$ introduced in~\cite{Spingarn1981SubmonotoneSO,rockafellar1998}, 
with equivalent definitions in~\cite{daniilidis2004}. 
A function $r:O\to \Rbb$, where $O\subset\Rbb^n$ is open, is said to be lower-$C^k$ on $O$, if on some neighborhood $V$ of each $\bar{x}\in O$ there is a representation
\begin{equation}\label{eqn:lowc1-def-0}
\centering
 \begin{aligned}
	 r(x) =  \underset{\substack{t \in T}}{\text{max}} \ r_t(x),
 \end{aligned}
\end{equation}
where $r_t:\Rbb^n\to\Rbb$ is of class $C^k$ on $V$ and the index set $T$ is a compact space such that $r_t$ 
and all of its partial derivatives through order $k$ are jointly continuous on $(t,x)\in T\times V$. 
A function is called upper-$C^k$ if we replace the $max$ with $min$ in~\eqref{eqn:lowc1-def-0}.
Let $T\subset\Rbb^p$ be compact, the function $r$ is upper-$C^2$ if it can be expressed as
\begin{equation}\label{eqn:uppc2-def-1}
\centering
 \begin{aligned}
	 r(x) =  \underset{\substack{t \in T}}{\text{min}} \ p(t,x)
 \end{aligned}
\end{equation}
for all $x\in O$, such that $p:\Rbb^p\times\Rbb^n\to\Rbb$ and its first- and second-order partial derivatives in $x$ 
depend continuously on $(t,x)$. Clearly, upper-/lower-$C^k$ imply upper-/lower-regularity.
To expand on the upper-$C^2$ examples given in Section~\ref{se:intro},
immediately from its definition~\eqref{eqn:uppc2-def-1}, two-stage (stochastic) optimization problems that are coupled only in a smooth objective 
have upper-$C^2$ objective in the first-stage problem, regardless of how complex the feasible set for second-stage problem is. Many constraint-coupled second-stage problems can be relaxed to obtain an upper-$C^2$ solution function, using for example, quadratic penalty regularization~\cite{wang2022}. 


Further, lower- and upper-type properties have equivalent definitions based on function and subgradient values that are more useful in analysis. Specifically, the finite-valued function $r$ is lower-$C^2$ on $O\subset\Rbb^n$ if and only if there exists $\rho>0$ such that 
$r(\cdot)+\frac{1}{2}\rho\norm{\cdot}^2$ is convex. 
Notice that this definition is already given on $O$ with a uniform $\rho$. Through simple arithmetic and convexity with subgradients, 
for all $x\in O$
\begin{equation}\label{eqn:lowc2-def}
\centering
 \begin{aligned}
	 r(x) - r(\bar{x}) - \langle g,x-\bar{x}\rangle + \frac{\rho}{2}\norm{x-\bar{x}}^2\geq 0,  
 \end{aligned}
\end{equation}
for $g\in\bar{\partial}r(\bar{x})$.   
By the symmetry property of Clarke subgradient~\cite[Proposition~2.3.1]{clarke1983}, if $g\in\bar{\partial} r(\bar{x})$, 
then $-g\in\bar{\partial} (-r)(\bar{x})$. Hence, if $r$ is upper-$C^2$, by definition $-r$ is lower-$C^2$, 
and we have 
\begin{equation}\label{eqn:uppc2-def}
\centering
 \begin{aligned}
	 r(x) - r(\bar{x}) - \langle g,x -\bar{x} \rangle \leq \frac{\rho}{2}\norm{x-\bar{x}}^2, 
 \end{aligned}
\end{equation}
where $g \in \bar{\partial} r(\bar{x})$.
We refer to~\eqref{eqn:uppc2-def} as the upper-$C^2$ inequality. 
Another common assumption for nonsmooth nonconvex objective is prox-regularity~\cite{wang2019}. Given a Lipschitz $r$ on an open set $O\subset \Rbb^n$, if $r$ is lower-$C^2$, it is prox-regular. 

Next, we assume that the feasible set for $x$ is bounded, in line with applications such as power-grid optimization~\cite{petra_21_gollnlp}. To reflect this
assumption directly in problem formulation and to simplify notations, without loss of generality,~\eqref{eqn:opt0} is recast as
\begin{equation} \label{eqn:opt-ms-simp}
 \centering
  \begin{aligned}
   &\underset{\substack{x}}{\text{minimize}} 
	  & & r(x)\\
   &\text{subject to}
	  & & c(x) = 0, \\
	  &&& 0 \leq x \leq x_u,
  \end{aligned}
\end{equation}
where $x_u\in\Rbb^n$.
The dimension of the constraint is $m$, \textit{i.e.}, $c(x)\in \Rbb^m$.

There are multiple optimality conditions in nonsmooth nonconvex analysis, \textit{e.g.}, 
stationary point, Karush–Kuhn–Tucker (KKT) point, Fritz-John critical point. 
In this paper, we assume calmness for problem~\eqref{eqn:opt-ms-simp}~\cite[Section~6.4]{clarke1983} at its local minimum $\bar{x}$.  
Calmness can be viewed as a weak constraint qualification.
In particular, the widely adopted linear
independence constraint qualification (LICQ) 
ensures calmness. Calmness effectively validates the use of a KKT point for Clarke necessary optimality condition~\cite{clarke1983}, where the Lagrange  multiplier for the objective function is nonzero in the Fritz-John critical point equation. 
Hence, for problem~\eqref{eqn:opt-ms-simp}, a necessary first-order optimality condition at a local minimum $\bar{x}$ is that there exists $\bar{\lambda} \in \Rbb^m $,
$\bar{\zeta}_l \in \Rbb^n\geq 0$ and $\bar{\zeta_u} \in \Rbb^n\geq 0$ such that 
\begin{equation} \label{eqn:opt-ms-simp-KKT}
   \centering
  \begin{aligned}
	  0 \in  \bar{\partial} r(\bar{x}) + \nabla c (\bar{x}) \bar{\lambda} - \bar{\zeta}_l + \bar{\zeta}_u&, \\
	  \bar{Z}_u(\bar{x}-x_u) = 0, \bar{Z}_l \bar{x} = 0&,\\
	  c_j(\bar{x}) = 0,j=1,\dots,m&,\\
	  \bar{\lambda}_j c_j(\bar{x}) = 0,j=1,\dots,m&,\\
	  \bar{\zeta}_l,\bar{\zeta}_u,x_u-\bar{x},\bar{x}\geq 0&.
  \end{aligned}
\end{equation}
The matrix $\nabla c(\bar{x})$ is of dimension $n\times m$.
The matrices $\bar{Z}_u,\bar{Z}_l$ are diagonal matrices whose diagonal values are $\bar{\zeta}_u$ and $\bar{\zeta}_l$, respectively.
A point that satisfies~\eqref{eqn:opt-ms-simp-KKT} is called a KKT point of~\eqref{eqn:opt-ms-simp}.
For upper regular functions, it is possible to establish a stronger form 
of subgradient optimality condition as explained in~\cite{mordukhovich2004upp}.

\section{\normalsize Simplified bundle algorithm}\label{sec:alg}
Given the nonsmooth nonconvex objective $r(\cdot)$ in~\eqref{eqn:opt-ms-simp}, in this section we consider bundle methods, which have 
proven to be one of the most successful methods in solving such 
problems~\cite{lemarechal2001}. 
The proposed algorithm simplifies the bundle while retain many of the algorithm's features.
Motivated by the discussion in Section~\ref{sec:prob}, we make the following assumption. 
\begin{assumption}\label{assp:upperC2}
	The Lipschitz continuous objective $r(\cdot)$ in~\eqref{eqn:opt-ms-simp} is upper-$C^2$. 
\end{assumption}
\noindent In particular, inequality~\eqref{eqn:uppc2-def} is satisfied.
We point out that~\cite{noll2009,dao2015} have explored bundle methods and to our knowledge were the first to prove convergence for upper-$C^1$
objective and constraints. However, in their case the parameters of the approximation model are not guaranteed to be finite, in addition to the aforementioned challenges in applying bundle method to constrained optimization.
We assume uniform boundedness on the Hessian of the constraints, a common assumption in literature~\cite{curtis2012}.
\begin{assumption}\label{assp:boundedHc}
	The function $c(\cdot)$ is twice differentiable and there exists a constant $H_u^c\geq 0$ such that $\frac{1}{2} |x^T \nabla^2c_j(x) x | \leq H_u^c\norm{x}^2$ for all feasible $x$ and all $j\in\{1, 2, \ldots, m\}$. 
\end{assumption}
Assumption~\ref{assp:boundedHc} can be relaxed to $c(\cdot)$ being continuously differentiable with Lipschitz continuous first-order derivative~\cite[Chapter 9]{rockafellar1998}. We also call $c(\cdot)$ smooth if Assumption~\ref{assp:boundedHc} is satisfied.
\subsection{Algorithm description}\label{sec:alg-1}
The simplified bundle algorithm is an iterative method with locally approximated objective at each iteration. 
It bears similarity to SQP in the treatment of constraints and can be viewed as its extension.
Compared to conventional bundle methods, the local convex quadratic approximation $\phi_k(\cdot)$ to $r(\cdot)$  
is dependent only on the current serious point instead of a bundle of points. 
More specifically, at iteration $k$ and its serious step $x_k$, $\phi_k(\cdot)$ is
\begin{equation} \label{eqn:opt-rc-appx-x}
 \centering
  \begin{aligned}
	  \phi_k(x) = r(x_k) + g_k^T(x-x_k) + \frac{1}{2}\alpha_k \norm{x-x_k}^2,
  \end{aligned}
\end{equation}
where $g_k\in\bar{\partial} r(x_k)$, and $\alpha_k>0$ is a scalar quadratic coefficient.
Equivalently, denoting $d=x-x_k$, $\phi_k(x)$ can be reformulated as $\Phi_k(d)$ such that
\begin{equation} \label{eqn:opt-rc-appx}
 \centering
  \begin{aligned}
	  \Phi_k(d) &= r_k + g_k^T d +\frac{1}{2}\alpha_k\left\lVert d\right\rVert ^2,\\ 
  \end{aligned}
\end{equation}
where $r_k=r(x_k)$. The function 
value and subgradient at $x_k$ are exact, \textit{i.e.}, $\Phi_k(0)=r_k,\nabla\Phi_k(0)=g_k$.
Furthermore, the constraints in~\eqref{eqn:opt-ms-simp} are linearized. The subproblem to be solved 
at iteration $k$ is 
\begin{equation} \label{eqn:opt-ms-simp-bundle}
 \centering
  \begin{aligned}
   &\underset{\substack{d}}{\text{minimize}} 
	  & & \Phi_k(d)\\
   &\text{subject to}
	  & & c(x_k) + \nabla c(x_k)^T d= 0, \ \ \ d_l^k \leq d \leq d_u^k,
  \end{aligned}
\end{equation}
where $d_l^k=-x_k,d_u^k=x_u-x_k$.
As in SQP methods, it is possible that the linearized constraints 
in~\eqref{eqn:opt-ms-simp-bundle} are infeasible. 
We address this possibility in Section~\ref{sec:lincons}.
In this section, we assume that the solution $d_k$ to~\eqref{eqn:opt-ms-simp-bundle} can be computed.
To measure progress in both the objective and constraints, 
a $l_1$ merit function is adopted in the form of
\begin{equation} \label{eqn:opt-ms-simp-bundle-merit}
 \centering
  \begin{aligned}
	  \phi_{1\theta_k}(x) = r(x) + \theta_k \norm{c(x)}_1, \\
  \end{aligned}
\end{equation}
where $\norm{\cdot}_1$ is the 1-norm and $\theta_k>0$ is a penalty parameter. 
A line search step on the constraints is needed in order to ensure progress
in the merit function~\eqref{eqn:opt-ms-simp-bundle-merit}.
The predicted change on the objective is defined as 
\begin{equation} \label{def:pd}
 \centering
  \begin{aligned}
	  \delta_k =& \Phi_k(0) - \Phi_k(d_k) 
	   = -g_k^T d_k -\frac{1}{2}\alpha_k\norm{d_k}^2.
  \end{aligned}
\end{equation}
To measure whether the approximation model $\Phi_k(\cdot)$ of the objective formed at $x_k$ is still valid at the trial step $x_k+d_k$,
we define the ratio $\rho_k$ as
\begin{equation} \label{eqn:decrease-ratio-1}
 \centering
  \begin{aligned}
	  \rho_k = \begin{cases}
		  r(x_k)-r(x_k+d_k)-\eta_l^+ \delta_k, \ &\delta_k \geq 0,\\
		  r(x_k)-r(x_k+d_k)-\eta_l^- \delta_k, \  &\delta_k <0,
	  \end{cases}
  \end{aligned}
\end{equation}
where $0<\eta_l^+\leq1$ and $\eta_l^-\geq 1$ are two constant parameters of the algorithm. 
If $\rho_k>0$, the model is still valid and the algorithm proceeds to line search.
Otherwise, the trial step $x_k+d_k$ is 
rejected and the parameter $\alpha_k$ is updated to find a different trial step. This process 
draws inspiration from trust-region methods and is a key element in many bundle methods. 

The value of $\delta_k$ is not necessarily positive. Therefore, the 
corresponding threshold $\eta_l^+$ and $\eta_l^-$ differ based on the sign of $\delta_k$. 
In both cases, the actual change in the objective $r(x_k)-r(x_{k}+d_k)$ is allowed to
be slightly worse than the predicted one. 
Namely, if $\delta_k\geq 0$, then the actual decrease can be smaller than the predicted decrease $\delta_k$, 
though a fraction $\eta_l^+$ of  $\delta_k$ is required;
and when $\delta_k<0$, 
the actual increase in objective value can be slightly larger than the predicted increase value $-\delta_k$, 
the degree to which is governed by $\eta_l^- \geq 1$.

Let the line search step size be $\beta_k \in (0,1]$. Then, the serious step taken is given as 
$x_{k+1} = x_k + \beta_k d_k$.
By  letting $\delta_k^{\beta} = \Phi_k(0) - \Phi_k(\beta_k d_{k})$, we have  
\begin{equation} \label{def:pd2}
 \centering
  \begin{aligned}
	  \delta_k^{\beta} =& \Phi_k(0) - \Phi_k(\beta_k d_k)
	   = -\beta_k g_k^T d_k -\frac{1}{2}\beta_k^2 \alpha_k\norm{d_k}^2.
  \end{aligned}
\end{equation}
The predicted and actual change in objective at $x_{k+1}$ is measured by $\rho_k^{\beta}$ defined by
\begin{equation} \label{eqn:decrease-ratio-beta}
 \centering
  \begin{aligned}
	  \rho_k^{\beta} = \begin{cases}
		  r(x_k)-r(x_{k+1})-\eta_{\gamma}^+ \delta_k^{\beta}, \ & \delta_k^{\beta} \geq 0,\\
		  r(x_k)-r(x_{k+1})-\eta_{\gamma}^- \delta_k^{\beta} , \  & \delta_k^{\beta} <0.
	  \end{cases}
  \end{aligned}
\end{equation}
 The parameter $\eta_{\gamma}^{+}\in(0,1]$ and $\eta_{\gamma}^-\geq 1$ can have different values than $\eta_l^+$ and $\eta_l^-$ to increase  
 flexibility of the algorithm.
The first-order optimality conditions of subproblem~\eqref{eqn:opt-ms-simp-bundle} are
\begin{equation} \label{eqn:simp-bundle-KKT}
  \centering
   \begin{aligned}
	   g_k + \alpha_k d_k - \nabla c(x_k) \lambda^{k+1} -\zeta_l^{k+1}+\zeta_u^{k+1} =0,&\\
	   Z_u^{k+1}(d_k-d_u^k) = 0, Z_l^{k+1} (d_k-d_l^k) = 0,&\\
	   \lambda_j^{k+1}\left[c_j(x_k)+\nabla c_j(x_k)^T d_k\right] = 0,j=1,\dots,m,&\\
	   \zeta_u^{k+1},\zeta_l^{k+1},d_k -d_l^k,d_u^k-d_k\geq 0,&\\
	   c(x_k)+\nabla c(x_k)^Td_k= 0,&
   \end{aligned}
 \end{equation}
	where $\lambda^{k+1}\in \Rbb^m$ is the Lagrange multiplier for constraints $c$ and $\zeta_u^{k+1}$ and $\zeta_l^{k+1}\in \Rbb^n$ are the Lagrange multipliers for the bound constraints. 
	The matrices $Z_u^{k+1}$ and $Z_l^{k+1}$ are 
	diagonal matrices with values $\zeta_u^{k+1}$ and $\zeta_l^{k+1}$, respectively.
      An equivalent form of the complementarity conditions 
      of bound constraints based on $x_u$  is
\begin{equation} \label{eqn:simp-bundle-KKT-bound}
  \centering
   \begin{aligned}
	   Z_u^{k+1}(x_k+d_k-x_u) = 0, Z_l^{k+1} (x_k+d_k) &= 0,\\
	   \zeta_u^{k+1},\zeta_l^{k+1}, x_k+d_k ,x_u-x_k-d_k&\geq 0.\\
   \end{aligned}
 \end{equation}
\noindent The line search conditions are given as follows
\begin{equation} \label{eqn:line-search-cond}
 \centering
  \begin{aligned}
	  \theta_k \norm{c(x_k)}_1+\beta_k (\lambda^{k+1})^T c(x_k) &\geq \theta_k \norm{c(x_{k+1})}_1-\eta_{\beta}\frac{1}{2}\alpha_k\beta_k\norm{d_k}^2, \\
	  \theta_k \norm{c(x_k)}_1+\eta_{\gamma}^+ \beta_k (\lambda^{k+1})^T c(x_k) &\geq \theta_k \norm{c(x_{k+1})}_1-\eta_{\beta}\frac{1}{2}\alpha_k\beta_k\norm{d_k}^2, \\
	  \theta_k \norm{c(x_k)}_1+\eta_{\gamma}^- \beta_k (\lambda^{k+1})^T c(x_k) &\geq \theta_k \norm{c(x_{k+1})}_1-\eta_{\beta}\frac{1}{2}\alpha_k\beta_k\norm{d_k}^2. \\
  \end{aligned}
\end{equation}
The differences between the conditions are the parameters $\eta_{\gamma}^+$ and $\eta_{\gamma}^-$ in the second and third inequalities, which stem from the unknown sign of $\delta_k$ and $\delta_k^{\beta}$. 
For simplicity in implementation and analysis, we use the following alternative 
condition 
\begin{equation} \label{eqn:line-search-cond-alt}
 \centering
  \begin{aligned}
	  \theta_k \norm{c(x_k)}_1 - \eta_{\gamma}^-  \beta_k\left| (\lambda^{k+1})^T c(x_k)\right| &\geq \theta_k \norm{c(x_{k+1})}_1-\eta_{\beta}\frac{1}{2}\alpha_k\beta_k\norm{d_k}^2. \\
  \end{aligned}
\end{equation}
We will show that condition~\eqref{eqn:line-search-cond-alt} implies conditions in~\eqref{eqn:line-search-cond} in Lemma~\ref{lem:line-search-alt}.
The simplified bundle method is presented in Algorithm~\ref{alg:simp-bundle}, where $\norm{\cdot}_{\infty}$ 
is the infinity norm. The terms for consistency restoration such as $\pi_{k-1}$ are explained in Section~\ref{sec:lincons}.
\begin{algorithm}
   \caption{Simplified bundle method}\label{alg:simp-bundle}
	\begin{algorithmic}[1]
	    \STATE{Initialize $x_0$, $\alpha_0$, stopping error tolerance $\epsilon$, and $k=1$.
	Choose scalars $0<\eta_l^+\leq 1$, $0<\eta_{\beta}<\eta_{\gamma}^+\leq 1$, $\eta_l^-\geq 1$,$\eta_{\gamma}^-\geq 1$, $\eta_{\alpha}>1$ and $\gamma>0$. 
	    Evaluate the function value $r(x_0)$ and subgradient $g_0$.}
	\FOR{$k=0,1,2,...$}
	  \STATE{Form the quadratic function $\Phi_k$ in~\eqref{eqn:opt-rc-appx} and solve 
		subproblem~\eqref{eqn:opt-ms-simp-bundle} to obtain $d_{k}$ and Lagrange multiplier $\lambda^{k+1}$. (If constraints are inconsistent, enter consistency restoration Algorithm~\ref{alg:simp-bundle-const}. Then, go back to step 2 with $k=k+1$.)}
	  \IF{$\norm{d_k}\leq \epsilon$}
	    \STATE{Stop the iteration and exit the algorithm.}
	  \ENDIF
	  \STATE{Evaluate function value $r(x_k+d_k)$. 
	    Compute $\delta_k$ in~\eqref{def:pd} and $\rho_k$ in~\eqref{eqn:decrease-ratio-1}.}
		\STATE{Set $\theta_k$ in~\eqref{eqn:opt-ms-simp-bundle-merit} with $\theta_k = \max{\{\theta_{k-1},\eta_{\gamma}^- \norm{\lambda^{k+1}}_{\infty}+\gamma\}}$. If Algorithm~\ref{alg:simp-bundle-const} is called for iteration $k-1$, let $\theta_k = \max{\{\pi_{k-1},\eta_{\gamma}^- \norm{\lambda^{k+1}}_{\infty}+\gamma\}}$.}
	  \IF{$\rho_k > 0$}
		  \STATE{Find the line search step size $\beta_k>0$ using backtracking, starting at $\beta_k=1$ and reducing by half 
		  if too large, such that the conditions in~\eqref{eqn:line-search-cond-alt} are satisfied.
		  Evaluate $r(x_{k+1})$ and compute $\rho_k^{\beta}$ in~\eqref{eqn:decrease-ratio-beta}.}
		  \IF{$\rho_k^{\beta} < 0$}
		    \STATE{Break and go to line 14.}
		  \ENDIF
		  \STATE{Take the step $x_{k+1} = x_k+\beta_k d_k$.}
	  \ELSE
		  \STATE{Reject the trial step.}
	    \STATE{Call the chosen $\alpha_k$ update rules to obtain $\alpha_{k+1}=\eta_{\alpha}\alpha_k$.}
	  \ENDIF
       \ENDFOR
    \end{algorithmic}
\end{algorithm}



\subsection{Convergence analysis}\label{sec:alg-convg}
If the algorithm terminates in a finite number of steps, the stopping test at step 4 is satisfied with the error tolerance $\epsilon$ and the solution is considered found. 
Let $\epsilon=0$,  based on step 4, $\norm{d_k}= 0$.
Since $d_k$ solves~\eqref{eqn:opt-ms-simp-bundle}, optimality conditions in~\eqref{eqn:simp-bundle-KKT}
are satisfied, of which the first equation reduces to 
\begin{equation}\label{alg:finite-steps}
 \begin{aligned}
	 g_k- \nabla c(x_k) \lambda^{k+1}-\zeta_l^{k+1}+\zeta_u^{k+1}=0.
 \end{aligned}
\end{equation}
Given $g_k\in\bar{\partial} r(x_k)$, we have $0\in \bar{\partial} r(x_k)-\nabla c(x_k) \lambda^{k+1}-\zeta_l^{k+1}+\zeta_u^{k+1}$. 
In addition, by $c(x_k)+\nabla c(x_k)^T d_k = 0$ from~\eqref{eqn:simp-bundle-KKT}, we have $c(x_k)=0$. Thus, $x_k$ is feasible in terms of $c$. 
Since the bound constraints are enforced in the subproblem~\eqref{eqn:opt-ms-simp-bundle}, the rest of the equations in~\eqref{eqn:opt-ms-simp-KKT} are also satisfied.
Therefore,  $x_k$ satisfies~\eqref{eqn:opt-ms-simp-KKT} and is by definition a KKT point for~\eqref{eqn:opt-ms-simp} as the algorithm exits.
In what follows, the analysis is focused on the case with an infinite number of steps, \textit{i.e.}, $\norm{d_k}>0$. 
We start by showing that the parameter $\alpha_k$ in Algorithm~\ref{alg:simp-bundle} eventually stabilizes, \textit{i.e.}, becomes constant.


\begin{lemma}\label{lem:sufficientdecrease}
	Given Assumption~\ref{assp:upperC2}, Algorithm~\ref{alg:simp-bundle} produces a finite 
	number of rejected steps. 
	As a consequence, the quadratic coefficient $\alpha_k$  is bounded above and stays constant for $k$ large enough. 
\end{lemma}
\begin{proof}
	From the upper-$C^2$ inequality~\eqref{eqn:uppc2-def}, we have  
\begin{equation} \label{eqn:opt-ms-appx-rec-p1}
 \centering
  \begin{aligned}
	  r(x_k+d)- r_k - g_{k}^Td \leq C \norm{d}^2,\\
  \end{aligned}
\end{equation}
for a fixed constant $C>0$ on the bounded domain of $x$.
In the first part of the proof we show that if at some iteration $k$, $\alpha_k$ satisfies
\begin{equation} \label{eqn:opt-ms-appx-rec-p2}
 \centering
  \begin{aligned}
	  \alpha_k > 2 C, \\
  \end{aligned}
\end{equation}
then no rejected steps can occur in Algorithm~\ref{alg:simp-bundle} after iteration $k$, or equivalently, $\rho_t> 0$ and $\rho_t^\beta> 0$ for all iterations $t\geq k$.
The inequalities~\eqref{eqn:opt-ms-appx-rec-p1} and \eqref{eqn:opt-ms-appx-rec-p2} imply
\begin{equation} \label{eqn:opt-ms-appx-rec-obj}
 \centering
  \begin{aligned}
	  r_k - r(x_k+d_k) \geq& -g_{k}^T d_k -C \norm{d_k}^2\\
                        >& -g_{k}^T d_k -\frac{1}{2}\alpha_k \norm{d_k}^2
				 = \Phi_k(0) -\Phi_k(d_k). \\
  \end{aligned}
\end{equation}
	As in the definition~\eqref{eqn:decrease-ratio-1} of $\rho_k$, we distinguish between two cases based on the sign of $\delta_k$.
	If $\delta_k = \Phi_k(0) -\Phi_k(d_k) \geq 0$, then since $0<\eta_l^+\leq 1$,~\eqref{eqn:opt-ms-appx-rec-obj} gives  
\begin{equation} \label{eqn:opt-ms-appx-rec-obj2}
 \centering
  \begin{aligned}
	  r_k - r(x_k+d_k) >& \Phi_k(0) -\Phi_k(d_k) 
			\geq  \eta_l^+\left[ \Phi_k(0) -\Phi_k(d_k)\right]. 
  \end{aligned}
\end{equation}
If $\delta_k = \Phi_k(0) -\Phi_k(d_k) < 0$, given $\eta_l^-\geq 1$, we can also write based on~\eqref{eqn:opt-ms-appx-rec-obj} that  
\begin{equation} \label{eqn:opt-ms-appx-rec-obj3}
 \centering
  \begin{aligned}
	  r_k - r(x_k+d_k) >& \Phi_k(0) -\Phi_k(d_k) 
			\geq  \eta_l^-\left[ \Phi_k(0) -\Phi_k(d_k)\right]. 
  \end{aligned}
\end{equation}
As a consequence, by definition~\eqref{eqn:decrease-ratio-1} we conclude 
 $\rho_k > 0$. Similar inequalities hold for  $x_{k+1} = x_k+\beta_k d_k$ since one can write based on~\eqref{eqn:opt-ms-appx-rec-p1} that
\begin{equation} \label{eqn:opt-ms-appx-rec-merit-beta}
 \centering
  \begin{aligned}
	  r(x_k) - r(x_{k+1}) \geq& -\beta_k g_{k}^Td_k - C \beta_k^2 \norm{d_k}^2 
                                 > -\beta_k g_{k}^Td_k -\frac{1}{2}\alpha_k\beta_k^2 \norm{d_k}^2 \\
				 =& \Phi_k(0) -\Phi_k(\beta_k d_k).\\
  \end{aligned}
\end{equation}
Same reasoning that leads to~\eqref{eqn:opt-ms-appx-rec-obj2} and~\eqref{eqn:opt-ms-appx-rec-obj3} for $\eta_{\gamma}^+$ and $\eta_{\gamma}^-$ implies that  $\rho_k^{\beta}> 0$. Therefore, for all $t\geq k$, one has $\rho_t>0, \rho_t^{\beta}>0$ and thus $\alpha_t =\alpha_k$. 
Since Algorithm~\ref{alg:simp-bundle} increases $\alpha_k$ monotonically with a ratio $\eta_{\alpha}>1$ whenever a rejected step is encountered,
only a finite number of rejected steps are needed to reach $\alpha_k > 2 C$ and is followed by an infinite number of serious steps. 

For the second part of the proof, suppose now $\alpha_k \leq 2C$ for all $k$. 
This is only possible when no or a finite number of rejected steps have been taken by the algorithm followed by all serious steps. This completes the proof. 
\end{proof}

\begin{remark}\label{rmrk:realalpha}
For simplicity, $\alpha_k$ is increased monotonically in the algorithm. In practice, we encourage that 
$\alpha_k$ be reduced if $\rho_k>0$ and $\eta_l^+>\eta_u^+$  where $\eta_u^+$ is an upper threshold for $\eta_l^+$.
In other words, if the actual decrease in objective is bigger than a certain ratio of the predicted decrease, 
	then $\Phi_k(\cdot)$ is a good approximation and we reduce the quadratic coefficient to encourage larger step size.
This adaptation is for when the upper-$C^2$ constant $C$ is not uniform in the entire domain. 
A decrease in $\alpha_k$ allows the algorithm to adjust better to  
	the local upper-$C^2$ constant that could be relatively small versus $C$, which could improve convergence.
	In addition, one can use a diagonal matrix instead of a scalar $\alpha_k$ to accelerate convergence without compromising the sparse structure of the linear system~\cite{wang2022}.
\end{remark}

\begin{lemma}\label{lem:line-search-merit}
	Given Assumption~\ref{assp:boundedHc}, if the	Lagrange multipliers $\lambda^{k+1}$ of~\eqref{eqn:opt-ms-simp-bundle} is bounded, the line search process of Algorithm~\ref{alg:simp-bundle} is well-defined, in that 
	there exists $\beta_k\in (0, 1]$ that satisfies the line search conditions in~\eqref{eqn:line-search-cond-alt} and can be found in a finite number of steps through backtracking step 9 .
\end{lemma}
\begin{proof}
If $\lambda^{k+1}$ is bounded throughout the algorithm, then a finite and constant $\theta_k$ is guaranteed as well for $k$ large enough based on how it is chosen in Algorithm~\ref{alg:simp-bundle} step 7.
	Since constraints $c$ are twice differentiable, we apply Taylor expansion to the $j$th equality constraint, $j=1,\dots,m$, at $x_k$ for $x_{k+1}=x_k+\beta_k d_k$ to obtain 
   \begin{equation} \label{eqn:simp-bundle-c-ls-pf-1}
   \centering
    \begin{aligned}
	    c_j(x_{k+1})  = & c_j(x_k)+ \beta_k \nabla c_j(x_k)^T d_{k} + \frac{1}{2}\beta_k^2 d_k^T H_{k\beta}^{j} d_k,
   \end{aligned}
   \end{equation}
	where $H_{k\beta}^j$ is the Hessian $\nabla^2 c_j(\cdot)$ at a point on the line segment between $x_k$ and $x_{k+1}$. 
Given $d_k$ as the solution to~\eqref{eqn:opt-ms-simp-bundle}, we have that $c_j(x_k) +  \nabla c_j(x_k)^T d_k = 0$ and, as a consequence, we can write based on~\eqref{eqn:simp-bundle-c-ls-pf-1} that
   \begin{equation*} 
    \centering
    \begin{aligned}
	    c_j(x_{k+1}) = (1-\beta_k) c_j(x_k)+ \frac{1}{2}\beta_k^2 d_k^T H_{k\beta}^{j} d_k.\\
    \end{aligned}
   \end{equation*}
By Assumption~\ref{assp:boundedHc}, $\left|c_j(x_{k+1})\right|  \leq \left|(1-\beta_k) c_j(x_k)\right| + \beta_k^2 H^{c}_u \norm{d_k}^2$, which implies that
   \begin{equation} \label{eqn:simp-bundle-c-ls-pf-2.5}
   \centering
    \begin{aligned}
	    \norm{c(x_{k+1})}_1  \leq 
	     (1-\beta_k) \norm{c(x_k)}_1 + m\beta_k^2 H^{c}_u \norm{d_k}^2.
    \end{aligned}
   \end{equation}
On the other hand, simple norm inequalities imply
\begin{equation} \label{eqn:simp-bundle-c-ls-pf-norm}
   \centering
    \begin{aligned}
	    \beta_k \left|(\lambda^{k+1})^T c(x_k)\right| \leq \beta_k\norm{\lambda^{k+1}}_{\infty} \norm{c(x_k)}_1.
    \end{aligned}
   \end{equation}
	Since step 7 of the algorithm chooses $\theta_k \geq \eta_{\gamma}^- \norm{\lambda^{k+1}}_{\infty}+\gamma$, where $\eta_{\gamma}^-$ and $\gamma$ are positive
	constants, we can write based on~\eqref{eqn:simp-bundle-c-ls-pf-2.5} and~\eqref{eqn:simp-bundle-c-ls-pf-norm} that
   \begin{equation} \label{eqn:simp-bundle-c-ls-pf-3} 
   \centering
    \begin{aligned}
	    &\theta_k\norm{c(x_k)}_1 -\eta_{\gamma}^- \beta_k \left|(\lambda^{k+1})^T c(x_k)\right|-\theta_k \norm{c(x_{k+1})}_1  \\ 
	    &\hspace{1.1cm}\geq(\theta_k - \eta_{\gamma}^- \beta_k\norm{\lambda^{k+1}}_{\infty}) \norm{c(x_k)}_1 -\theta_k(1-\beta_k) \norm{c(x_k)}_1 - \theta_km\beta_k^2 H^{c}_u \norm{d_k}^2\\
		&\hspace{1.1cm}=(\theta_k\beta_k- \eta_{\gamma}^- \beta_k\norm{\lambda^{k+1}}_{\infty}) \norm{c(x_k)}_1- \theta_km\beta_k^2 H^c_u  \norm{d_k}^2 \\
		&\hspace{1.1cm}\geq\beta_k \gamma \norm{c(x_k)}_1- \theta_k m \beta_k^2H^c_u\norm{d_k}^2.
    \end{aligned}
   \end{equation}
Therefore, if $\beta_k$ is reduced through the backtracking of Algorithm~\ref{alg:simp-bundle} to satisfy
   \begin{equation} \label{eqn:simp-bundle-c-ls-pf-4}
   \centering
    \begin{aligned}
	   0< \beta_k \leq \frac{\eta_{\beta} \alpha_k}{2 H^c_u\theta_k m},
    \end{aligned}
   \end{equation}
then~\eqref{eqn:line-search-cond-alt} is satisfied.
Using ceiling function $\lceil\cdot \rceil$, which returns the least integer greater than the input, we can write 
   \begin{equation} \label{eqn:simp-bundle-c-ls-pf-5}
   \centering
    \begin{aligned}
	    \beta_k \geq \frac{1}{2} ^{\lceil \log_{\frac{1}{2}} \frac{\eta_{\beta} \alpha_k}{2 H^c_u\theta_k m} \rceil}.
    \end{aligned}
   \end{equation}
We remark that both the denominator and numerator in~\eqref{eqn:simp-bundle-c-ls-pf-4} are positive and independent of the line search. 
Further, by Lemma~\ref{lem:sufficientdecrease} and boundedness of $\lambda^{k+1}$, all terms in~\eqref{eqn:simp-bundle-c-ls-pf-4} remain finite. Therefore, the backtracking stops in finite steps.
%
%
\end{proof}

\begin{lemma}\label{lem:line-search-alt}
	The $\beta_k \in (0,1]$ that meets the line search condition in~\eqref{eqn:line-search-cond-alt} also satisfies the conditions 
	from~\eqref{eqn:line-search-cond}, or equivalently
\begin{equation} \label{eqn:simp-bundle-etagamma}
   \centering
    \begin{aligned}
	   \beta_k(\lambda^{k+1})^T c(x_k)\geq -\theta_k \norm{c(x_k)}_1+ \theta_k \norm{c(x_{k+1})}_1 - \eta_{\beta}\frac{1}{2}\alpha_k\beta_k\norm{d_k}^2,\\
	  \eta_{\gamma}^+\beta_k(\lambda^{k+1})^T c(x_k) \geq  -\theta_k \norm{c(x_k)}_1+ \theta_k \norm{c(x_{k+1})}_1   - \eta_{\beta}\frac{1}{2}\alpha_k\beta_k\norm{d_k}^2,\\
	 \eta_{\gamma}^-\beta_k(\lambda^{k+1})^T c(x_k) \geq  -\theta_k \norm{c(x_k)}_1+ \theta_k \norm{c(x_{k+1})}_1 - \eta_{\beta}\frac{1}{2}\alpha_k\beta_k\norm{d_k}^2.\\
    \end{aligned}
   \end{equation}
\end{lemma}
\begin{proof}
   From the absolute value inequality, we have  
   \begin{equation} \label{eqn:line-search-alt-pf-1}
   \centering
    \begin{aligned}
	    \beta_k(\lambda^{k+1})^T c(x_k) \geq& -\beta_k\left|(\lambda^{k+1})^T c(x_k)\right|,\\ 
	     \eta_{\gamma}^+ \beta_k(\lambda^{k+1})^T c(x_k) \geq& -\eta_{\gamma}^+ \beta_k\left|(\lambda^{k+1})^T c(x_k)\right|,\\
	     \eta_{\gamma}^- \beta_k(\lambda^{k+1})^T c(x_k) \geq& -\eta_{\gamma}^-\beta_k\left|(\lambda^{k+1})^T c(x_k)\right|.\\
    \end{aligned}
   \end{equation}
	Given that $0<\eta_{\gamma}^{+}\leq 1\leq \eta_{\gamma}^-$, 
   \begin{equation} \label{eqn:line-search-alt-pf-2}
   \centering
    \begin{aligned}
            -\eta_{\gamma}^-\beta_k\left|(\lambda^{k+1})^T c(x_k)\right|\leq  -\beta_k\left|(\lambda^{k+1})^T c(x_k)\right| \leq -\eta_{\gamma}^+ \beta_k\left|(\lambda^{k+1})^T c(x_k)\right|.
    \end{aligned}
   \end{equation}
	Therefore, from~\eqref{eqn:line-search-alt-pf-1} and~\eqref{eqn:line-search-alt-pf-2}
   \begin{equation} \label{eqn:line-search-alt-pf-3}
   \centering
    \begin{aligned}
	    \beta_k(\lambda^{k+1})^T c(x_k) \geq& -\eta_{\gamma}^- \beta_k\left|(\lambda^{k+1})^T c(x_k)\right|,\\ 
	     \eta_{\gamma}^+ \beta_k(\lambda^{k+1})^T c(x_k) \geq& -\eta_{\gamma}^- \beta_k\left|(\lambda^{k+1})^T c(x_k)\right|,\\
	     \eta_{\gamma}^- \beta_k(\lambda^{k+1})^T c(x_k) \geq& -\eta_{\gamma}^-\beta_k\left|(\lambda^{k+1})^T c(x_k)\right|.\\
    \end{aligned}
   \end{equation}
   From the line search condition~\eqref{eqn:line-search-cond-alt}, we have
   \begin{equation} \label{eqn:line-search-alt-pf-4}
   \centering
    \begin{aligned}
	-\eta_{\gamma}^-\beta_k\left|(\lambda^{k+1})^T c(x_k)\right| \geq -\theta_k \norm{c(x_k)}_1+ \theta_k \norm{c(x_{k+1})}_1 - \eta_{\beta}\frac{1}{2}\alpha_k\beta_k\norm{d_k}^2.\\
    \end{aligned}
   \end{equation}
   Combined with~\eqref{eqn:line-search-alt-pf-3} the proof is completed.
\end{proof}

\begin{lemma}\label{lem:sufficientdecrease-merit}
	The serious step $x_{k+1} = x_k+\beta_k d_k$ is a decreasing step for the merit function~\eqref{eqn:opt-ms-simp-bundle-merit} if $\beta_k$ is obtained via line search. Moreover, if $\lambda^k$ is bounded, the speed of decrease satisfies $\phi_{1\theta_k}(x_k)-\phi_{1\theta_k}(x_{k+1})>c_{\phi} \norm{d_k}^2$ for some $c_{\phi}>0$. 
\end{lemma}
\begin{proof}
	For a serious step $x_{k+1}$ to be taken, step 8 and 10 give us: $\rho_k>0,\rho_k^{\beta}>0$. 
	We distinguish three cases based on the value of $\alpha_k$ and sign of $\delta_k^{\beta}$. 
	The first case is $\alpha_k>2C$. By upper-$C^2$ property in~\eqref{eqn:opt-ms-appx-rec-p1}, as shown in~\eqref{eqn:opt-ms-appx-rec-merit-beta}, we have  
\begin{equation*} 
 \centering
  \begin{aligned}
	  r(x_k) - r(x_{k+1}) >& -\beta_k g_{k}^Td_k -\frac{1}{2}\alpha_k\beta_k^2 \norm{d_k}^2. \\
  \end{aligned}
\end{equation*}
	In the second case, 
	$\alpha_k\leq 2C$ and $\delta_k^{\beta}\geq 0 $.
	From the definition of $\rho_k^{\beta}$ in~\eqref{eqn:decrease-ratio-beta},  
\begin{equation} \label{eqn:simp-bundle-merit-pf-2p}
 \centering
  \begin{aligned}
	  r(x_k) - r(x_{k+1}) > 
		  \eta_{\gamma}^+ \left[ -\beta_k g_{k}^Td_k -\frac{1}{2}\alpha_k\beta_k^2 \norm{d_k}^2\right]. \\
  \end{aligned}
\end{equation}
       The third case is when $\alpha_k\leq 2C$ and $\delta_k^{\beta}<0$, for which 
\begin{equation} \label{eqn:simp-bundle-merit-pf-2m}
 \centering
  \begin{aligned}
	  r(x_k) - r(x_{k+1}) > 
		  \eta_{\gamma}^- \left[ -\beta_k g_{k}^Td_k -\frac{1}{2}\alpha_k\beta_k^2 \norm{d_k}^2\right]. \\
  \end{aligned}
\end{equation}
		Rearranging the first equation in optimality conditions~\eqref{eqn:simp-bundle-KKT}, we have  
\begin{equation} \label{eqn:simp-bundle-KKT-beta}
	   g_k + \alpha_k d_k =  \nabla c(x_k) \lambda^{k+1} +\zeta_l^{k+1} -\zeta_u^{k+1}.
 \end{equation}
 By taking the inner product with $-d_k$ and using the last equation from~\eqref{eqn:simp-bundle-KKT}, we have 
\begin{equation} \label{eqn:simp-bundle-KKT-2}
  \centering
   \begin{aligned}
	   - g_k^T d_k-\alpha_k\norm{d_k}^2 &=- (\lambda^{k+1})^T \nabla c(x_k)^T d_k-d_k^T \zeta_l^{k+1}+d_k^T \zeta_u^{k+1}\\
			    &= (\lambda^{k+1})^T c(x_k) - (d_k - d_l^k +d_l^k)^T\zeta_l^{k+1} + (d_k-d_u^k+d_u^k)^T \zeta_u^{k+1} \\
			    &= (\lambda^{k+1})^T c(x_k) - (d_l^k)^T \zeta_l^{k+1}  +(d_u^k)^T \zeta_u^{k+1} \\
			    &= (\lambda^{k+1})^T c(x_k) + x_k^T \zeta_l^{k+1} + (x_u-x_k)^T\zeta_u^{k+1} \\
			    &\geq (\lambda^{k+1})^T c(x_k). \\
   \end{aligned}
 \end{equation}
	The third equality of~\eqref{eqn:simp-bundle-KKT-2} comes from $Z_{l}^{k+1}(d_k-d_l^k)=0$ and $Z_u^{k+1}(d_k-d_u^k)=0$ in~\eqref{eqn:simp-bundle-KKT}. 
	The inequality is obtained from bound constraints in~\eqref{eqn:simp-bundle-KKT-bound} 
	where $x_k\geq 0$, $x_u-x_k\geq 0$, $\zeta_l^{k+1}\geq 0$, and $\zeta_u^{k+1}\geq 0$
	for the current and previous iterations. 
	Next, multiplying both sides of~\eqref{eqn:simp-bundle-KKT-2} by $\beta_k$ and then subtracting $\frac{1}{2}\alpha_k\beta_k^2\norm{d_k}^2$ leads to
  \begin{equation} \label{eqn:simp-bundle-KKT-3}
       \centering
       \begin{aligned}
	       -\beta_kg_k^T d_k-\frac{1}{2}\alpha_k\beta_k^2\norm{d_k}^2 &\geq \alpha_k\beta_k\norm{d_k}^2-\frac{1}{2}\alpha_k\beta_k^2\norm{d_k}^2+\beta_k(\lambda^{k+1})^Tc(x_k)\\
		 &\geq\frac{1}{2}\alpha_k\beta_k\norm{d_k}^2+\beta_k(\lambda^{k+1})^Tc(x_k),\\
         \end{aligned}
        \end{equation}
	where the second inequality makes use of $\beta_k\in (0,1]$.
Since the left-hand side of~\eqref{eqn:simp-bundle-KKT-3} is $\delta_k^{\beta}$,  
by multiplying both sides of~\eqref{eqn:simp-bundle-KKT-3} by $\eta_{\gamma}^+$ and $\eta_{\gamma}^-$ respectively, we obtain
  \begin{align} 
	       -\eta_{\gamma}^+ \beta_kg_k^T d_k-\frac{1}{2}\eta_{\gamma}^+\alpha_k\beta_k^2\norm{d_k}^2 
		 &\geq\frac{1}{2}\eta_{\gamma}^+\alpha_k\beta_k\norm{d_k}^2+\eta_{\gamma}^+\beta_k(\lambda^{k+1})^Tc(x_k), \label{eqn:simp-bundle-KKT-4}\\
	       -\eta_{\gamma}^- \beta_kg_k^T d_k-\frac{1}{2}\eta_{\gamma}^-\alpha_k\beta_k^2\norm{d_k}^2 
		 &\geq\frac{1}{2}\eta_{\gamma}^-\alpha_k\beta_k\norm{d_k}^2+\eta_{\gamma}^-\beta_k(\lambda^{k+1})^Tc(x_k).\label{eqn:simp-bundle-KKT-5}
\end{align}
Finally, we examine the merit function $\phi_{1\theta_k}(\cdot)$. If $\alpha_k>2C$, 
we combine the inequality in~\eqref{eqn:opt-ms-appx-rec-merit-beta},~\eqref{eqn:simp-bundle-KKT-3}, and the first inequality from Lemma~\ref{lem:line-search-alt} to write
\begin{equation} \label{eqn:simp-bundle-merit-pf-3}
 \centering
  \begin{aligned}
	  \phi_{1\theta_k}(x_k) - \phi_{1\theta_k}(x_{k+1})  &= r(x_k) -r(x_{k+1})+\theta_k\norm{c(x_k)}_1-\theta_k\norm{c(x_{k+1})}_1\\
			 >& -\beta_kg_k^T d_k-\frac{1}{2}\alpha_k\beta_k^2\norm{d_k}^2+\theta_k\norm{c(x_k)}_1-\theta_k\norm{c(x_{k+1})}_1 \\
			 \geq&\frac{1}{2}\alpha_k\beta_k\norm{d_k}^2+\beta_k(\lambda^{k+1})^Tc(x_k)+\theta_k\norm{c(x_k)}_1-\theta_k\norm{c(x_{k+1})}_1\\
			 \geq&\frac{1}{2}\alpha_k\beta_k\norm{d_k}^2-\eta_{\beta}\beta_k\frac{1}{2}\alpha_k\norm{d_k}^2
			 = (1-\eta_{\beta})\frac{1}{2}\alpha_k\beta_k\norm{d_k}^2.
  \end{aligned}
\end{equation}
Similarly, if $\alpha_k\leq 2C$ and $\delta_k^{\beta}\geq 0$, by applying in order~\eqref{eqn:simp-bundle-merit-pf-2p},~\eqref{eqn:simp-bundle-KKT-4} and the second inequality from Lemma~\ref{lem:line-search-alt}, we write 
\begin{equation} \label{eqn:simp-bundle-merit-pf-4}
 \centering
  \begin{aligned}
	  \phi_{1\theta_k}(x_k) - \phi_{1\theta_k}&(x_{k+1})  = r(x_k) -r(x_{k+1})+\theta_k\norm{c(x_k)}_1-\theta_k\norm{c(x_{k+1})}_1\\
			 >& -\eta_{\gamma}^+\beta_kg_k^T d_k-\eta_{\gamma}^+\frac{1}{2}\alpha_k\beta_k^2\norm{d_k}^2+\theta_k\norm{c(x_k)}_1-\theta_k\norm{c(x_{k+1})}_1 \\
			 \geq&\frac{1}{2}\eta_{\gamma}^+\alpha_k\beta_k\norm{d_k}^2+\eta_{\gamma}^+\beta_k(\lambda^{k+1})^Tc(x_k)+\theta_k\norm{c(x_k)}_1-\theta_k\norm{c(x_{k+1})}_1\\
			 \geq&\frac{1}{2}\eta_{\gamma}^+\alpha_k\beta_k\norm{d_k}^2-\eta_{\beta}\beta_k\frac{1}{2}\alpha_k\norm{d_k}^2
			 = (\eta_{\gamma}^+-\eta_{\beta})\frac{1}{2}\alpha_k\beta_k\norm{d_k}^2.
  \end{aligned}
\end{equation}
Recall that $\eta_{\gamma}^+-\eta_{\beta}>0$ by design of Algorithm~\ref{alg:simp-bundle}. Finally, when $\delta_k^{\beta}\leq 0$, applying
in order~\eqref{eqn:simp-bundle-merit-pf-2m},~\eqref{eqn:simp-bundle-KKT-5} and the third inequality from Lemma~\ref{lem:line-search-alt}, we have 
\begin{equation} \label{eqn:simp-bundle-merit-pf-5}
 \centering
  \begin{aligned}
	  \phi_{1\theta_k}(x_k) - \phi_{1\theta_k}&(x_{k+1})  =r(x_k) -r(x_{k+1}^{\beta})+\theta_k\norm{c(x_k)}_1-\theta_k\norm{c(x_{k+1})}_1\\
			 >& -\eta_{\gamma}^-\beta_kg_k^T d_k-\eta_{\gamma}^-\frac{1}{2}\alpha_k\beta_k^2\norm{d_k}^2+\theta_k\norm{c(x_k)}_1-\theta_k\norm{c(x_{k+1})}_1 \\
			 \geq&\frac{1}{2}\eta_{\gamma}^-\alpha_k\beta_k\norm{d_k}^2+\eta_{\gamma}^-\beta_k(\lambda^{k+1})^Tc(x_k)+\theta_k\norm{c(x_k)}_1-\theta_k\norm{c(x_{k+1})}_1\\
			 \geq&\frac{1}{2}\eta_{\gamma}^-\alpha_k\beta_k\norm{d_k}^2-\eta_{\beta}\beta_k\frac{1}{2}\alpha_k\norm{d_k}^2
			 = (\eta_{\gamma}^- -\eta_{\beta})\frac{1}{2}\alpha_k\beta_k\norm{d_k}^2,
  \end{aligned}
\end{equation}
where $\eta_{\gamma}^- -\eta_{\beta}>0$.
Therefore, in all cases, a serious step $x_{k+1}=x_k+\beta_kd_k$ is a 
decreasing direction for $\phi_{1\theta_k}(\cdot)$. 
If $\lambda^{k}$ is bounded, $\theta_k$ will stay constant for $k$ large enough from step 7 of Algorithm~\ref{alg:simp-bundle}. Let $\bar{\theta}$ be the constant value so that $\theta_k \leq \bar{\theta}$ for all $k$.  
Then, by~\eqref{eqn:simp-bundle-c-ls-pf-5},
\begin{equation} \label{eqn:simp-bundle-merit-pf-6}
 \centering
  \begin{aligned}
	  \beta_k \geq \frac{1}{2} ^{\lceil \log_{\frac{1}{2}} \frac{\eta_{\beta} \alpha_0}{2 H^c_u\bar{\theta} m} \rceil}:=\bar{\beta},
  \end{aligned}
\end{equation}
due to the monotonicity of $\alpha_k$ and $\theta_k$. In other words, $\beta_k$ is bounded below by $\bar{\beta}$ for all $k$.
From~\eqref{eqn:simp-bundle-merit-pf-3},~\eqref{eqn:simp-bundle-merit-pf-4},~\eqref{eqn:simp-bundle-merit-pf-5}, $\phi_{1\theta_k}(x_k)-\phi_{1\theta_k}(x_{k+1})>(\eta_{\gamma}^+-\eta_{\beta})\frac{1}{2}\alpha_0 \bar{\beta}\norm{d_k}^2$. 
Or simply, there exists $c_{\phi}>0$ such that $\phi_{1\theta_k}(x_k)-\phi_{1\theta_k}(x_{k+1})>c_{\phi}\norm{d_k}^2$. 
\end{proof}
In order to attain bounded Lagrange multipliers, a constraint qualification is necessary 
for $c(\cdot)$ in~\eqref{eqn:opt-ms-simp}.
Given the existence of both equality and inequality constraints, we resort to LICQ~\cite{Nocedal_book}. A topic of further research will be to obtain global convergence under a weaker constraint qualification such as the Slater condition. 

\begin{lemma}\label{lem:bounded-lp}
	If LICQ of the constraints in~\eqref{eqn:opt-ms-simp} are satisfied at every accumulation points $\bar{x}$ of serious steps $\{x_k\}$ generated by the algorithm, then 
	the sequence of Lagrange multipliers for problem~\eqref{eqn:opt-ms-simp-bundle} $\{\zeta_u^{k+1}\},\{\zeta_l^{k+1}\}$ and $\{\lambda^{k+1}\}$ are bounded. 
	In addition, this means there exists $k$ such that $\theta_t = \theta_k$ for all $t\geq k$.
\end{lemma}
\begin{proof}
     We rewrite the first equation in optimality condition in~\eqref{eqn:simp-bundle-KKT} as
    \begin{equation} \label{eqn:simp-bundle-KKT-full}
     \centering
     \begin{aligned}
	     g_k + \alpha_k d_k - \sum_{j=1}^m \lambda^{k+1}_j \nabla c_j(x_k) -\sum_{i=1}^n (\zeta_l^{k+1})_i e_i 
	     +\sum_{i=1}^n (\zeta_u^{k+1})_i e_i &=0,\\
     \end{aligned}
    \end{equation}
    where $e_i \in \Rbb^n$ is a vector such that $(e_i)_i = 1$ and $(e_i)_l = 0,l\neq i$. 
    The Lagrange multipliers for the bound constraints can be combined into one vector $\zeta^{k+1}=\zeta_l^{k+1}- \zeta_u^{k+1}$,
    since $(\zeta_l^{k+1})_i(\zeta_u^{k+1})_i = 0$.
	A component of $\zeta_l^{k+1}$ or $\zeta_u^{k+1}$ is unbounded if and only if the corresponding component in $\zeta^{k+1}$ is unbounded.
	Let $I$ be the index set of the active bound constraints, we have
    \begin{equation} \label{eqn:simp-bundle-KKT-full-2}
     \centering
     \begin{aligned}
	     g_k + \alpha_k d_k = \sum_{j=1}^m \lambda^{k+1}_j \nabla c_j(x_k) +\sum_{i \in I} (\zeta^{k+1})_i e_i.\\
     \end{aligned}
    \end{equation}
	Since $\{x_k\},\{g_k\}$ are bounded ($r$ being Lipschitz on a bounded domain) and $\{\alpha_k\}$ is finite by Lemma~\ref{lem:sufficientdecrease}, 
	the left-hand side of the equation stays bounded as $k\to\infty$. 
	From LICQ at $\bar{x}$, $\nabla c_j(\bar{x}) \in \Rbb^n $ and $e_i,i \in I$ are linearly independent and bounded vectors. 
	Without losing generality, suppose $\lambda^{k+1}_j,j\in[1,m]$ is not 
	bounded as $k\to\infty$. Then, $\norm{\lambda^{k}}_{\infty} \to \infty$ and we can construct an unbounded subsequence $\{\lambda^{k_u}\}$ that is monotonic in $\lambda^{k_u}_j$. The corresponding subsequences $\{x_{k_u}\}$, $\{g_{k_u}\}$, $\{\alpha_{k_u}\}$ remain bounded. Passing on further to a subsequence if necessary, in which case the subsequence of $\{\lambda^{k_u}\}$ is still unbounded, we can assume $x_{k_u}\to\bar{x}$ as $k\to\infty$, where $\bar{x}$ is an accumulation point. Regardless of the behavior of $\{\zeta^{k_u}\}$, the right-hand side of~\eqref{eqn:simp-bundle-KKT-full-2} will be unbounded due to linear independence at $\bar{x}$ among the vectors.
	This is a contradiction. Same process can be repeated for $\zeta^{k+1}_j,j\in[1,m]$. 
	Therefore, there exist $\lambda^U\geq 0$, $\zeta_l^U\geq 0$, $\zeta_u^U\geq 0$  such that 
	$\norm{\lambda^{k}}_{\infty} \leq \lambda^U$, $\zeta^{k}_u \leq \zeta_u^U$ and $\zeta^{k}_l \leq \zeta_l^U$ for all $k$.
	Since $\theta_k$ is determined by $\lambda^k$, there exists $k$ such that $\theta_t=\theta_k$ for all $t\geq k$. 
\end{proof}

\begin{theorem}\label{thm:simp-KKT}
	Given the Assumptions~\ref{assp:upperC2} and~\ref{assp:boundedHc}, if the constraints in~\eqref{eqn:opt-ms-simp} satisfy the 
	conditions in Lemma~\ref{lem:bounded-lp}, then every accumulation point $\bar{x}$ of the serious steps $\{x_k\}$ generated by Algorithm~\ref{alg:simp-bundle} 
	is a KKT point of the problem~\eqref{eqn:opt-ms-simp}. 
	That is, there exists a subsequence $\{x_{k_s}\}$ of $\{x_k\}$, where $ x_{k_s} \to \bar{x}$, and $\bar{\lambda}\in\Rbb^m$, $\bar{\zeta}_u\in\Rbb^n$, $\bar{\zeta}_l\in\Rbb^n$, so that the first-order optimality conditions~\eqref{eqn:opt-ms-simp-KKT} are satisfied at $\bar{x}$.
\end{theorem}
\begin{proof}
	By Lemma~\ref{lem:sufficientdecrease}, there exists $k_0>0$ such that for all $t\geq k_0$, $\alpha_t=\alpha_{k_0}=\bar{\alpha}$ and 
	all steps are serious steps.
	By Lemma~\ref{lem:bounded-lp}, there exists $k_1>0$ such that for $t\geq k_1$, the Lagrange multipliers are bounded above and $\theta_t=\theta_{k_1}=\bar{\theta}$. 
	We say $k$ is large enough if $k\geq \max{(k_0,k_1)}$, in which case the parameters of the algorithm stabilizes at $\alpha_t=\bar{\alpha}$ and $\theta_t=\bar{\theta}$ for $t\geq k$.

	Since the domain of $x$ is bounded and $r(\cdot)$ is Lipschitz, the serious steps sequence $\{x_k\}$ 
	as well as the subgradient sequence $\{g_{k}\}$ are bounded. 
	Therefore, there exists at least one accumulation point for $\{x_k\}$.
	Let $\bar{x}$ be an accumulation point of $\{x_k\}$ and $\{x_{k_s}\}$ be a subsequence of $\{x_k\}$ such that $x_{k_s}\to \bar{x}$.
	From Lemma~\ref{lem:line-search-merit}, line search terminates successfully and  
	by Lemma~\ref{lem:sufficientdecrease-merit}, for $k$ large enough, $\{\phi_{1\theta_k}(x_k)\}$ is a decreasing and bounded sequence 
	with a fixed parameter $\bar{\theta}$. 
	Thus, $\phi_{1\theta_k}(x_k)$ converges. 
	From Lemma~\ref{lem:sufficientdecrease-merit},
	we know that $\phi_{1\theta_k}(x_k)-\phi_{1\theta_k}(x_{k+1})$ is 
	bounded below in the order of $\norm{d_k}^2$. Therefore,  $\lim_{k\to\infty} \norm{d_k}\to 0$. In particular, $\lim_{s\to \infty}\norm{d_{k_s}}\to 0$.
	By the last equation in~\eqref{eqn:simp-bundle-KKT}, $c(x_{k_s})\to 0$. Thus, $\bar{x}$ satisfies the equality constraints $c(\bar{x})=0$.
        Given that the bound constraints are satisfied by all $x_k$, we have $0\leq \bar{x}\leq x_u$. 
          
	Passing on to a subsequence if necessary, we let $g_{k_s}\to \bar{g}$, 
	$\lambda^{k_{s}+1} \to \bar{\lambda}$, $\zeta_u^{k_{s}+1} \to \bar{\zeta}_u$, $\zeta_l^{k_{s}+1} \to \bar{\zeta}_l$. 
        From the first equation in the optimality conditions~\eqref{eqn:simp-bundle-KKT}, we have 
        \begin{equation} \label{eqn:simp-bundle-KKT-limit-1}
          \centering
          \begin{aligned}
		  0 = \bar{g} - \nabla c(\bar{x}) \bar{\lambda}  -\bar{\zeta}_l +\bar{\zeta}_u. 
          \end{aligned}
        \end{equation}
	By the outer semicontinuity of Clarke subdifferential, with $g_{k_s}\in \bar{\partial} r(x_{k_s})$, we have $\bar{g} \in \bar{\partial} r(\bar{x})$.
        As a result, 
	  $0 \in \bar{\partial} r(\bar{x}) -\nabla c(\bar{x}) \bar{\lambda} -\bar{\zeta}_l +\bar{\zeta}_u$.
	The complementarity conditions of bound constraints from~\eqref{eqn:simp-bundle-KKT} lead 
	to $\bar{Z}_u(\bar{x}-x_u)$, $\bar{Z}_l\bar{x}=0$. Together with the 
	equality constraints $c(\bar{x}) = 0$, the first-order necessary 
	optimality conditions~\eqref{eqn:opt-ms-simp-KKT} of problem~\eqref{eqn:opt-ms-simp} at $\bar{x}$ are satisfied.
\end{proof}

While the line search is only conducted on the less computationally expensive and analytically known 
function $c$, it is 
possible to avoid it altogether. For example, the bundle method in~\cite{hare2015} does not need a line search when the feasible set for $x$ is convex.
%
Similar results stand for the simplified bundle algorithm.  
\begin{proposition}\label{prop:convex-constraint}
	If the equality and bound constraints in~\eqref{eqn:opt-ms-simp} form a convex and bounded set in $\Rbb^n$, 
	then instead of~\eqref{eqn:opt-ms-simp-bundle} we solve subproblem
        \begin{equation} \label{eqn:opt-ms-simp-bundle-2}
        \centering
         \begin{aligned}
          &\underset{\substack{x}}{\text{minimize}} 
	  & & \phi_k(x)\\
          &\text{subject to}
	  & & c(x) = 0, ~0 \leq x\leq x_u.
  \end{aligned}
\end{equation}
	In this case, the line search step can be replaced with $x_{k+1} = x_k +d_k$, where $d_k$ is the solution to~\eqref{eqn:opt-ms-simp-bundle-2}
	and the global convergence results of this section still apply.
\end{proposition}

The algorithm and convergence analysis can be readily extended to two-stage stochastic programming problems in the form of 
~\eqref{eqn:opt0}, 
where the quadratic approximation function $\phi_k(\cdot)$ is needed only for second-stage solution function $r$. Details of the setup and additional update rules for $\alpha_k$
can be found in the report~\cite{wang2022}.

\subsection{Consistency restoration of linearized constraints}\label{sec:lincons}
As mentioned previously, the linearized constraints of the subproblem~\eqref{eqn:opt-ms-simp-bundle} can become infeasible even when the original problem~\eqref{eqn:opt-ms-simp} is feasible, a phenomenon known as inconsistency. To address this possibility, this section introduces and analyzes a consistency restoration algorithm. Whenever the subproblem~\eqref{eqn:opt-ms-simp-bundle} has inconsistent linearized constraints, a penalty subproblem (see~\eqref{eqn:opt-ms-simp-bundle-penal} below) that incorporates a measure of constraints' violation in the objective is solved. The salient idea is to generate a new serious point with \textit{consistent} linearized constraints. 
For problems satisfying constraint qualification such as LICQ, the penalty problem with a well-designed penalty parameter could generate feasible accumulation points.
For problems that do not, it is possible to minimize the linearized constraint violation for infeasible accumulation points through the update rule of the penalty parameter.
For example, one can use the approach in~\cite{byrd2005} that requires solving a so-called feasibility problem.
%
%
%
We first formulate the following penalty subproblem :
  \begin{align}
   &\underset{\substack{d}}{\text{min}} 
	  & &  \Phi_k(d) + \pi_k \norm{c(x_k)+\nabla c(x_k)^T d}_1\\ 
	  & \text{s.t.}
	  & & d_l^k\leq d\leq d_u^k, \label{eqn:opt-ms-simp-bundle-penal-nonsmooth}
  \end{align}
where $\pi_k > 0$ is the penalty parameter. 
 To avoid difficulties with the nonsmooth 1-norm, the algorithm solves the following equivalent quadratic programming problem:
\begin{equation} \label{eqn:opt-ms-simp-bundle-penal}
 \centering
  \begin{aligned}
   &\underset{\substack{d,v,w}}{\text{min}} 
	  & & \Phi_k(d) +\pi_k \sum_{j=1}^m (v_j + w_j)  \\
   &\text{s.t.}
	  & & c_j(x_k)+\nabla c_j(x_k)^T d = v_j - w_j, \ j=1,\ldots,m,\\
	  &&& d_l^k\leq d\leq d_u^k, \ \ 0 \leq v, w ,
  \end{aligned}
\end{equation}
where $v,w\in \Rbb^m$ are slack variables. Let $d_k$, $v^k$, and $w^k$ denote the solutions to~\eqref{eqn:opt-ms-simp-bundle-penal}.
The first-order optimality conditions of problem~\eqref{eqn:opt-ms-simp-bundle-penal} for $d$ are
\begin{equation} \label{eqn:simp-bundle-penal-KKT-1}
  \centering
   \begin{aligned}
	   g_k + \alpha_k d_k + \sum_{i=1}^m \lambda^{k+1}_j \nabla c_j(x_k)-\zeta_l^{k+1}+\zeta_u^{k+1} =0&,\\
	   \lambda^{k+1}_j \left[c_j(x_k)+\nabla c_j(x_k)^T d_k - v^k_j + w^k_j \right] =0, \ j=1,...m,\\
	   c_j(x_k)+\nabla c_j(x_k)^T d_k - v^k_j + w^k_j =0, \ j=1,...m,\\
	   Z_u^{k+1}(d_k-d_u^k) = 0, Z_l^{k+1}(d_k-d_l^k) = 0&,\\
	   \zeta_u^{k+1},\zeta_l^{k+1},d_k+x_k,x_u-x_k-d_k\geq 0&.\\
   \end{aligned}
 \end{equation}
Here, $\lambda^{k+1}\in\Rbb^m$, $\zeta_u^{k+1}\in \Rbb^n$, and $\zeta_l^{k+1}\in \Rbb^n$ are the Lagrange multipliers for the constraints, while $Z_u^{k+1}$ and $Z_l^{k+1}$ are diagonal matrices with values $\zeta_u^{k+1}$ and $\zeta_l^{k+1}$, respectively.
The remaining optimality conditions on slack variables $v$ and $w$ are 
\begin{equation} \label{eqn:simp-bundle-penal-KKT-2}
  \centering
   \begin{aligned}
	   \pi_k-\lambda^{k+1}_j - p^{k+1}_j =0, \ j=1,...m,\\
           \pi_k+\lambda^{k+1}_j-q^{k+1}_j=0, \ j=1,...m,\\
	   P^{k+1} v^k  = 0, Q^{k+1} w^k = 0&,\\
	   v^k,w^k,p^{k+1},q^{k+1} \geq 0&,\\
   \end{aligned}
 \end{equation}
where $p^{k+1}\in\Rbb^m$ and $q^{k+1}\in\Rbb^m$ are the Lagrange multipliers for bound constraints on $v^k$ and $w^k$, respectively, and $P^{k+1}$ and $Q^{k+1}$ are diagonal matrices whose diagonal entries are elements of $p^{k+1}$ and $q^{k+1}$, respectively. 
Based on whether the slack variable bound constraints are active, the relations among the multipliers 
can be simplified. 
To proceed, we define the sign function $\sigma_j^k:\Rbb^n\to\Rbb$, $j=1,\dots,m$ of $d$ by
\begin{equation} \label{eqn:opt-ms-simp-sign}
 \centering
  \begin{aligned}
	  \sigma_j^k(d)  =  
	  \begin{cases}
		  -1,  
		       & c_j(x_k)+\nabla c_j(x_k)^T d < 0, \\  
		  \phantom{-}0,
		  & c_j(x_k) + \nabla c_j(x_k)^T d = 0,\\
		  \phantom{-}1,
		  & c_j(x_k) + \nabla c_j(x_k)^T d > 0.
    \end{cases}
  \end{aligned}
\end{equation}
In addition, for $d_k$, we define
the active equality constraint set as
\begin{equation} \label{eqn:opt-ms-simp-bundle-fea-A}
   \centering
    \begin{aligned}
     A_k = \{ 1\leq j\leq m|c_j(x_k)+\nabla c_j(x_k)^T d_{k} = 0\},
     \end{aligned}
\end{equation}
and the inactive equality constraint set as
\begin{equation} \label{eqn:opt-ms-simp-bundle-fea-I}
  \centering
    \begin{aligned}
      V_k = \{ 1\leq j\leq m|c_j(x_k)+\nabla c_j(x_k)^T d_{k} \neq 0\}.
    \end{aligned}
\end{equation}
  The optimality conditions~\eqref{eqn:simp-bundle-penal-KKT-2} are integrated into~\eqref{eqn:simp-bundle-penal-KKT-1}
in the following lemma.

\begin{lemma}\label{lem:lambda-relaxed-prop}
	For all $j\in V_k$, $\lambda_j^{k+1} = \sigma_j^k(d_k)\pi_k$. Also, for all $j\in A_k$, 
        $\lambda_j^{k+1}\in [-\pi_k,\pi_k]$.
\end{lemma}
\begin{proof}
	Note first that for any $ j\in[1,m]$, the slack variable solutions satisfy $v^k_j w^k_j = 0$. 
	We consider the three cases given by the value of 
	$\sigma_j^k(d_k),j=1,\dots,m$.
	If $\sigma_j^k(d_k)=1$, by the third equation in~\eqref{eqn:simp-bundle-penal-KKT-1}, $v_j^k>0$ and $w_j^k=0$.
	From~\eqref{eqn:simp-bundle-penal-KKT-2}, 
        $p^{k+1}_j=0$. Then, by the first equation in~\eqref{eqn:simp-bundle-penal-KKT-2}, $\lambda_j^{k+1}=\pi_k$.
	Similarly, if $\sigma_j^k(d_k)=-1$, 
        one obtains $q^{k+1}_j=0$ and $\lambda_j^{k+1}=-\pi_k$. The first part of the Lemma is proven.

	In the last case, $j\in A_k$, \textit{i.e.}, $\sigma_j^k(d_k)=0$. By the third equation in~\eqref{eqn:simp-bundle-penal-KKT-1}, we have $v_j^k=0$ and $w_j^k=0$.  
We sum and subtract the first two equations in~\eqref{eqn:simp-bundle-penal-KKT-2} to obtain 
		$\lambda_j^{k+1} = \frac{1}{2}(q_j^{k+1}-p_j^{k+1})$ and 
		$p_j^{k+1}+q_j^{k+1}=2\pi_k$. 
	The bound constraints on $p^{k+1}$ and $q^{k+1}$ and the latter equation implies $0\leq p_j^{k+1}$ and $q_j^{k+1}\leq 2\pi_k$; finally,
	the former equation implies $\lambda_j^{k+1}\in [-\pi_k,\pi_k]$.
\end{proof}

Similar to~\eqref{def:pd}, we define $\delta_k^{\pi_k}$ to be the predicted change in objective of the penalty subproblem~\eqref{eqn:opt-ms-simp-bundle-penal} with penalty $\pi_k$, which based on~\eqref{eqn:opt-ms-simp-bundle-penal-nonsmooth} is 
\begin{equation}\label{def:pd-penal}
   \begin{aligned}
	   \delta_k^{\pi_k} =& -g_k^T d_k -\frac{1}{2}\alpha_k \norm{d_k}^2 +\pi_k \norm{c(x_k)}_1 -\pi_k \norm{c(x_k)+\nabla c(x_k)^T d_k}_1.\\
   \end{aligned}
\end{equation}
Notice that $\delta_k^{\pi_k}\geq 0$. 
The ratios $\rho_k$~\eqref{eqn:decrease-ratio-1} and $\rho_k^{\beta}$~\eqref{eqn:decrease-ratio-beta} are again used to address the nonsmoothness of $r$. 
The algorithm also requires line search on the constraints to obtain a serious step 
$x_{k+1}=x_k+\beta_k d_k$, $\beta_k\in (0,1]$. 
To simplify the analysis, we adopt in this section $\eta_{\gamma}^-=\eta_{\gamma}^+=1$ and $\eta_l^-=\eta_l^+=1$, 
making the definition of $\rho_k,\rho_k^{\beta}$ in~\eqref{eqn:decrease-ratio-1} and~\eqref{eqn:decrease-ratio-beta} identical 
across branches.
Thus, regardless of the value of $\alpha_k$, 
a serious step $x_{k+1}$ satisfies 
$r(x_k)-r(x_{k+1}) > \Phi_k(0)-\Phi_k(\beta_kd_k)$ and $r(x_k)-r(x_{k}+d_k) > \Phi_k(0)-\Phi_k(d_k)$. 

The merit function and line search conditions for consistency restoration are 
\begin{equation}\label{def:merit-fea}
   \begin{aligned}
	   \phi_{1\pi_k}(x)  =  r(x) + \pi_k \norm{c(x)}_1,
   \end{aligned}
\end{equation}
and 
\begin{equation}\label{eqn:line-search-cond-penal}
   \begin{aligned}
   \centering
	    \norm{c(x_k)}_1  +  \frac{\beta_k}{\pi_k} (\lambda^{k+1})^T \nabla c(x_k)^T d_k  \geq
	    \norm{c(x_{k+1})}_1 -\frac{\eta_{\beta}}{\pi_k}\frac{1}{2}\alpha_k\beta_k\norm{d_k}^2. 
   \end{aligned}
\end{equation}

The consistency restoration algorithm is given in Algorithm~\ref{alg:simp-bundle-const}
where $c^+_k := \norm{c(x_k) + \nabla c(x_k)^T d_k}$. It activates in 
Algorithm~\ref{alg:simp-bundle} when inconsistency occurs at step 3 and exits after one serious step iteration in step 11 of Algorithm~\ref{alg:simp-bundle-const}. However, 
it is possible that the linearized constraints remain inconsistent and   
Algorithm~\ref{alg:simp-bundle-const} is called repeatedly. 
In this case, the update rule of the penalty parameter ensures that the penalty problem assigns a larger emphasis on the linearized constraint violations. 
\begin{algorithm}
   \caption{Simplified bundle method: consistency restoration}\label{alg:simp-bundle-const}
	\begin{algorithmic}[1]
		\STATE{Given $x_k$, $\alpha_k$, $r_k$, $g_k$, $\theta_k$ and $\epsilon$ from Algorithm~\ref{alg:simp-bundle}, choose coefficient $\gamma_f >0$ and constraint error tolerance $\epsilon_c\geq 0$.
		Let $\pi_k=\max{(\pi_{k-j},\theta_k)}, j\in[0,k]$, where $j$ is the smallest integer so that $\pi_{k-j}$ exists.}
	\STATE{ Solve~\eqref{eqn:opt-ms-simp-bundle-penal} with $\pi_k$ to obtain $d_k$ and $\lambda^{k+1}$.}
	\IF{$\delta_k \leq \epsilon$ and $c^+_k\leq\epsilon_c$ }
          \STATE{Stop the iteration and exit the algorithm.}
        \ENDIF
		\STATE{Set $\pi_{k+1} = \max\{\pi_k, \theta_k,\norm{\lambda^{k+1}}_{\infty}+\gamma_f  \}$.  }
	\STATE{ 
	    Evaluate $r(x_k+d_k)$, $\delta_k$ in~\eqref{def:pd} and $\rho_k$ in~\eqref{eqn:decrease-ratio-1}.}
	\IF{$\rho_k > 0$}
	  \STATE{Find the line search step size $\beta_k>0$ using backtracking, starting at $\beta_k=1$ and halving 
		  if too large, such that~\eqref{eqn:line-search-cond-penal} is satisfied. Compute $\rho_k^{\beta}$ in~\eqref{eqn:decrease-ratio-beta}.}
            \IF{$\rho_k^{\beta} < 0$}
	     \STATE{Break and go to 13.}
	    \ENDIF
	    \STATE{Take the serious step $x_{k+1} = x_k+\beta_k d_k$.
	    Exit consistency restoration. Go back to Algorithm~\ref{alg:simp-bundle} and start a new iteration.}
       \ELSE
	   \STATE{Reject the trial step and update $\alpha_k$ with $\alpha_{k+1}=\eta_{\alpha}\alpha_k$.
            Go back to step 1.}
       \ENDIF
    \end{algorithmic}
\end{algorithm}

While Algorithm~\ref{alg:simp-bundle-const} solves a penalty subproblem instead, 
it includes all the elements in Algorithm~\ref{alg:simp-bundle} to deal with the nonsmoothness of $r$, including the update rule for $\alpha_k$.
Thus, we can reuse many analyses from Section~\ref{sec:alg-convg}.
First, Lemma~\ref{lem:sufficientdecrease} holds true, as in the form of the following lemma.
\begin{lemma}\label{lem:sufficientdecrease-penal}
	Under the Assumption~\ref{assp:upperC2}, the consistency restoration Algorithm \ref{alg:simp-bundle-const} produces a finite number 
	of rejected steps. Consequently, the parameter $\alpha_k$ of Algorithm~\ref{alg:simp-bundle-const} stabilizes, \textit{i.e.}, there exists $k$ such that $\alpha_t = \alpha_k$ for all $t\geq k$. 
\end{lemma}
\begin{proof}
	Since Algorithm~\ref{alg:simp-bundle-const} has identical mechanism for rejecting steps and increasing $\alpha_k$ to Algorithm~\ref{alg:simp-bundle}, which only relies on the property of $r$, the proof of Lemma~\ref{lem:sufficientdecrease} can be directly applied here.
	Only a finite number of rejected steps are needed to achieve $\alpha_k>2C$, which guarantees $\rho_k>0$, $\rho_k^{\beta}>0$, and produces a serious step. If $\alpha_k\leq 2C$ for all $k$, then only finite number of rejected steps are generated, which also means all steps are serious steps, \textit{i.e.}, $\rho_k>0,\rho_k^{\beta}>0$, for $k$ large enough.
As a result, there exists a $k$ such that $\alpha_t=\alpha_k$ for $t\geq k$  (see proof of Lemma~\ref{lem:sufficientdecrease}), with a finite number of rejected steps produced.
\end{proof}

We next show that the line search step 8 in Algorithm~\ref{alg:simp-bundle-const} is well-defined.
\begin{lemma}\label{lem:line-search-merit-2}
	Given $\pi_k>0$ and Assumption~\ref{assp:boundedHc}, step 8 of Algorithm~\ref{alg:simp-bundle-const} finds $\beta_k\in (0,1]$ satisfying the line search condition~\eqref{eqn:line-search-cond-penal} in a finite number of steps.
\end{lemma}
\begin{proof}
	Since constraint functions $c$ are twice differentiable, by Taylor expansion of its $jth$ component, $1\leq j \leq m$, for $x_{k+1}=x_k+\beta_k d_k$,
       \begin{equation} \label{eqn:simp-bundle-fea-c-ls-pf-1}
         \centering
         \begin{aligned}
		 c_j(x_{k+1})  = & c_j(x_k)+ \beta_k \nabla c_j(x_k)^T d_{k} + \frac{1}{2}\beta_k^2 d_k^T H_{k\beta}^{j} d_k,\\
      \end{aligned}
     \end{equation}
	where $H_{k\beta}^j$ depends on both $x_k$ and $d_k$. 
	By Assumption~\ref{assp:boundedHc} and triangle inequality, we have
   \begin{equation} \label{eqn:simp-bundle-fea-c-ls-pf-2}
   \centering
    \begin{aligned}
	    |c_j(x_{k+1})|  &= | c_j(x_k)+ \beta_k \nabla c_j(x_k)^T d_{k}+ \frac{1}{2}\beta_k^2 d_k^T H_{k\beta}^{j} d_k|\\
	    &\leq | c_j(x_k) + \beta_k \nabla c_j(x_k)^T d_{k}| + \beta_k^2 H^{c}_u \norm{d_k}^2\\
	     &= \left|(1-\beta_k) c_j(x_k) + \beta_k \left[c_j(x_k)+ \nabla c_j(x_k)^T d_{k}\right]\right|+\beta_k^2 H^{c}_u \norm{d_k}^2\\ 
	      & \leq (1-\beta_k)| c_j(x_k)| + \beta_k |c_j(x_k)+ \nabla c_j(x_k)^T d_{k}|+ \beta_k^2 H^{c}_u \norm{d_k}^2 . 
    \end{aligned}
   \end{equation}
	Applying the definition~\eqref{eqn:opt-ms-simp-sign} of $\sigma_j^k$ to~\eqref{eqn:simp-bundle-fea-c-ls-pf-2}, we can write 
\begin{equation} \label{eqn:simp-bundle-fea-c-ls-pf-2.3}
   \centering
    \begin{aligned}
	    |c_j(x_{k+1})|  
		  & \leq (1-\beta_k)|c_j(x_k)| + \beta_k \sigma_j^k(d_k)\left[c_j(x_k)+ \nabla c_j(x_k)^T d_{k}\right] + \beta_k^2 H^{c}_u \norm{d_k}^2.
    \end{aligned}
   \end{equation}
	Let us denote the cardinality of $A_k$ and $V_k$ by $|A_k|$ and $|V_k|$, respectively. By summing up~\eqref{eqn:simp-bundle-fea-c-ls-pf-2.3} over $j\in V_k$ and noting that $|c_j(x_k)|\geq \sigma_j^k(d_k) c_j(x_k)$, we have
   \begin{equation} \label{eqn:simp-bundle-fea-c-ls-pf-2.5}
     \centering
      \begin{aligned}
	      \sum_{j\in V_k}|c_j(x_{k+1})|  \leq& 
	       \sum_{j \in V_k}  |c_j(x_k)|+\beta_k\sum_{j \in V_k} \sigma_j^k(d_k)\nabla c_j(x_k)^Td_k 
	      + |V_k|\beta_k^2 H^{c}_u \norm{d_k}^2.\\
      \end{aligned}
    \end{equation}
	Similarly, we sum up~\eqref{eqn:simp-bundle-fea-c-ls-pf-2.3} over $j\in A_k$ and use the definition in~\eqref{eqn:opt-ms-simp-bundle-fea-A} to write 
   \begin{equation} \label{eqn:simp-bundle-fea-c-ls-pf-2.6}
     \centering
      \begin{aligned}
\sum_{j\in A_k}|c_j(x_{k+1})|  \leq& 
	      (1-\beta_k)\sum_{j \in A_k}  |c_j(x_k)| + |A_k| \beta_k^2 H^{c}_u \norm{d_k}^2. \\
      \end{aligned}
    \end{equation}
    Summing the two equations in~\eqref{eqn:simp-bundle-fea-c-ls-pf-2.5} and~\eqref{eqn:simp-bundle-fea-c-ls-pf-2.6}
    and using $|A_k|+|V_k|=m$ gives us 
\begin{equation} \label{eqn:simp-bundle-fea-c-ls-pf-3}
     \centering
      \begin{aligned}
	      \norm{c(x_k)}_1-\norm{c(x_{k+1})}_1 \geq  
	    \beta_k  \sum_{j \in A_k}  |c_j(x_k)|   -\beta_k \sum_{j \in V_k} \sigma_j^k(d_k)\nabla c_j(x_k)^T d_k 
	     - m\beta_k^2 H^{c}_u \norm{d_k}^2. \\
      \end{aligned}
    \end{equation}
    From Lemma~\ref{lem:lambda-relaxed-prop} and the definition~\eqref{eqn:opt-ms-simp-bundle-fea-A} of $A_k$, we can write
\begin{equation} \label{eqn:simp-bundle-fea-c-ls-pf-3.1}
     \centering
      \begin{aligned}
	      \sum_{j=1}^m \lambda_j^{k+1}\nabla c_j(x_k)^T d_k =& \sum_{j\in V_k} \lambda_j^{k+1} \nabla c_j(x_k)^T d_k +\sum_{j\in A_k} \lambda_j^{k+1} \nabla c_j(x_k)^T d_k\\
	      =& \sum_{j\in V_k} \sigma_j^k(d_k)\pi_k \nabla c_j(x_k)^T d_k -\sum_{j\in A_k} \lambda^{k+1}_j  c_j(x_k).\\
	      \geq& \sum_{j\in V_k} \sigma_j^k(d_k)\pi_k \nabla c_j(x_k)^T d_k - \pi_k\sum_{j\in A_k}  |c_j(x_k)|.\\
      \end{aligned}
    \end{equation}
    The inequality in~\eqref{eqn:simp-bundle-fea-c-ls-pf-3.1} comes from the second part of Lemma~\ref{lem:lambda-relaxed-prop}, 
    namely $\lambda^{k+1}_j\in [-\pi_k,\pi_k]$.
	Through basic algebraic calculations and applying~\eqref{eqn:simp-bundle-fea-c-ls-pf-3} and~\eqref{eqn:simp-bundle-fea-c-ls-pf-3.1}, 
   \begin{equation} \label{eqn:simp-bundle-fea-c-ls-pf-3.5}
   \centering
    \begin{aligned}
	    &  \norm{c(x_k)}_1 - \norm{c(x_{k+1})}_1 + \frac{\beta_k}{\pi_k} \sum_{j=1}^m \lambda^{k+1}_j\nabla c_j(x_k)^T d_k
	     \geq \\
	    &\qquad \beta_k \sum_{j \in A_k}  |c_j(x_k)|  
	      -\beta_k  \sum_{j \in V_k} \sigma_j^k(d_k)\nabla c_j(x_k)^T d_k
	      -  m\beta_k^2 H^{c}_u \norm{d_k}^2 \\ 
	    & \qquad+ \beta_k\sum_{j\in V_k} \sigma_j^k(d_k) \nabla c_j(x_k)^T d_k
	     - \beta_k \sum_{j\in A_k} | c_j(x_k) | \geq 
	    -  m \beta_k^2 H^{c}_u \norm{d_k}^2.
    \end{aligned}
   \end{equation}
   Thus, if $\beta_k$ satisfies 
	    $0< \beta_k \leq \frac{\eta_{\beta} \alpha_k }{2m H^c_u \pi_k}$,
   where both the denominator and numerator are positive and independent of line search, the line search condition~\eqref{eqn:line-search-cond-penal} is
   satisfied.
Using ceiling function $\lceil\cdot \rceil$, we have 
   \begin{equation} \label{eqn:simp-bundle-fea-c-ls-pf-5}
   \centering
    \begin{aligned}
	    \beta_k \geq \frac{1}{2} ^{\lceil \log_{\frac{1}{2}} \frac{\eta_{\beta} \alpha_k}{2 m H^c_u \pi_k} \rceil}.
    \end{aligned}
   \end{equation}
	For each $k$, given a ${\pi_k} >0$, the line search exits after a finite number of steps. 
\end{proof}
One can prove the decrease in merit function similarly to Lemma~\ref{lem:sufficientdecrease-merit}. 
\begin{lemma}\label{lem:sufficientdecrease-merit-penal}
	The serious step $x_{k+1}=x_k+\beta_k d_k$ is a descent step for the merit function~\eqref{def:merit-fea} if 
	$\beta_k$ is obtained through step 8 in Algorithm~\ref{alg:simp-bundle-const}. Further, if $\pi_k$ stabilizes,
	\textit{i.e.}, there exists $k$ such that $\pi_t=\pi_k>0$ for all $t\geq k$, the reduction of the merit function 
	satisfies $\phi_{1\pi_k}(x_k)-\phi_{1\pi_k}(x_{k+1})>c_{\phi}^{\pi}\norm{d_k}^2$ for some $c_{\phi}^{\pi}>0$.
\end{lemma}
\begin{proof}
	Since  $\rho_k>0,\rho_k^{\beta}>0$ at any serious step,  and $\eta_l^+=\eta_l^-=\eta_{\gamma}^+=\eta_{\gamma}^-=1$, we can compactly write based on definitions~\eqref{eqn:decrease-ratio-1},~\eqref{eqn:decrease-ratio-beta} and~\eqref{eqn:opt-ms-appx-rec-merit-beta} that
\begin{equation} \label{eqn:simp-bundle-merit-penal-pf-1}
 \centering
  \begin{aligned}
	  r(x_k) - r(x_{k+1}) >& -\beta_k g_{k}^Td_k -\frac{1}{2}\alpha_k\beta_k^2 \norm{d_k}^2. \\
  \end{aligned}
\end{equation}
	Let us rearrange the first equation in the KKT conditions~\eqref{eqn:simp-bundle-penal-KKT-1} to obtain 
\begin{equation} \label{eqn:simp-bundle-penal-merit-pf-1}
  \centering
   \begin{aligned}
	    g_k + \alpha_k d_k  = -\sum_{j=1}^m\lambda_j^{k+1} \nabla c_j(x_k) 
	     + \zeta_l^{k+1}-\zeta_u^{k+1}.\\
   \end{aligned}
 \end{equation}
	Taking the dot product with $-d_k$ on both sides of~\eqref{eqn:simp-bundle-penal-merit-pf-1} leads to
\begin{equation} \label{eqn:simp-bundle-penal-merit-pf-2}
  \centering
   \begin{aligned}
	   - g_k^Td_k - \alpha_k\norm{d_k}^2 =& \sum_{j=1}^m\lambda^{k+1}_j \nabla c_j(x_k)^T d_k 
	    -d_k^T \zeta_l^{k+1} + d_k^T \zeta_u^{k+1}.\\
   \end{aligned}
 \end{equation}
	Recall that $d_l^k=-x_k$ and $d_u^k=x_u-x_k$.
	Using the fourth and fifth equations in~\eqref{eqn:simp-bundle-penal-KKT-1}, we have
\begin{equation} \label{eqn:simp-bundle-penal-merit-pf-2.3}
  \centering
   \begin{aligned}
	   -d_k^T \zeta_l^{k+1} + d_k^T \zeta_u^{k+1} =& -(d_k-d_l^k+d_l^k)^T\zeta_l^{k+1}+(d_k-d_u^k+d_u^k)^T\zeta_u^{k+1}\\
					       =& x_k^T \zeta_l^{k+1} +(x_u-x_k)^T \zeta_u^{k+1}
					       \geq 0.
   \end{aligned}
 \end{equation}
	The inequality above is obtained from the bound constraints, namely, $0\leq x_k\leq x_u$, $\zeta_u^{k+1}\geq 0$, and $\zeta_l^{k+1}\geq 0$.
	Applying~\eqref{eqn:simp-bundle-penal-merit-pf-2.3} to~\eqref{eqn:simp-bundle-penal-merit-pf-2}, we have
\begin{equation} \label{eqn:simp-bundle-penal-merit-pf-2.5}
  \centering
   \begin{aligned}
	   -g_k^Td_k - \alpha_k\norm{d_k}^2 
	   \geq& \sum_{j=1}^m \lambda_j^{k+1} \nabla c_j(x_k)^T d_k. 
   \end{aligned}
 \end{equation}
	Since $\beta_k\in (0,1]$, one can multiply by $\beta_k$ and
	subtract $\frac{1}{2}\alpha_k\beta_k^2\norm{d_k}^2$ from both sides of~\eqref{eqn:simp-bundle-penal-merit-pf-2.5} to obtain
\begin{equation} \label{eqn:simp-bundle-penal-merit-pf-3}
  \centering
   \begin{aligned}
	   -\beta_k g_k^Td_k -\frac{1}{2}\alpha_k\beta_k^2\norm{d_k}^2
	   \geq&  \beta_k \alpha_k\norm{d_k}^2 - \frac{1}{2}\alpha_k\beta_k^2\norm{d_k}^2 +\beta_k\sum_{j=1}^m \lambda_j^{k+1} \nabla c_j(x_k)^T d_k\\ 
	   \geq& \frac{1}{2}\alpha_k\beta_k\norm{d_k}^2 +\beta_k\sum_{j=1}^m \lambda_j^{k+1} \nabla c_j(x_k)^T d_k\\ 
	   =&  \frac{1}{2}\alpha_k\beta_k\norm{d_k}^2 +\beta_k  (\lambda^{k+1})^T \nabla c(x_k)^T d_k. 
   \end{aligned}
 \end{equation}
	On the other hand,~\eqref{eqn:simp-bundle-merit-penal-pf-1} together with~\eqref{eqn:simp-bundle-penal-merit-pf-3} imply that the merit function satisfies
\begin{equation*}
 \centering
  \begin{aligned}
	  \phi_{1\pi_k}(x_k) - \phi_{1\pi_k}(x_{k+1})  =& r(x_k) -r(x_{k+1})+ \pi_k\norm{c(x_k)}_1
		    -\pi_k\norm{c(x_{k+1})}_1\\
			 >& -\beta_kg_k^T d_k-\frac{1}{2}\alpha_k\beta_k^2\norm{d_k}^2+ \pi_k\norm{c(x_k)}_1
		    -\pi_k\norm{c(x_{k+1})}_1\\
			\geq \frac{1}{2}\alpha_k\beta_k\norm{d_k}^2& +\beta_k (\lambda^{k+1})^T \nabla c(x_k)^T d_k+ \pi_k\norm{c(x_k)}_1 - \pi_k\norm{c(x_{k+1})}_1. \\
  \end{aligned}
\end{equation*}
The above inequality and the line search condition~\eqref{eqn:line-search-cond-penal} lead to
\begin{equation*} \label{eqn:simp-bundle-penal-merit-pf-5}
 \centering
  \begin{aligned}
	  \phi_{1\pi_k}(x_k) - \phi_{1\pi_k}(x_{k+1})  
			 >\frac{1}{2}\alpha_k\beta_k\norm{d_k}^2-\eta_{\beta}\frac{1}{2}\alpha_k\beta_k\norm{d_k}^2
			 =(1-\eta_{\beta})\frac{1}{2}\alpha_k\beta_k\norm{d_k}^2,
  \end{aligned}
\end{equation*}
which proves that $x_{k+1}$ is a descent step for $\phi_{1\pi_k}(\cdot)$. In addition, let $\bar{\pi}$ be the stabilized value for $\pi_k$,~\eqref{eqn:simp-bundle-fea-c-ls-pf-5} implies
\begin{equation*}
 \centering
  \begin{aligned}
	  \beta_k \geq \frac{1}{2} ^{\lceil \log_{\frac{1}{2}} \frac{\eta_{\beta} \alpha_0 }{2m H^c_u \bar{\pi}} \rceil}:=\bar{\beta}^{\pi}.
  \end{aligned}
\end{equation*}
Thus, 
$\phi_{1\pi_k}(x_k)-\phi_{1\pi_k}(x_{k+1})> c_{\phi}^{\pi} \norm{d_k}^2$, where $c_{\phi}^{\pi} = (1-\eta_{\beta})\frac{1}{2}\alpha_0 \bar{\beta}^{\pi}>0$ for all $k$.
\end{proof}
The sequence of $\{\pi_k\}$ is by design monotonically non-decreasing. Similar to Lemma \ref{lem:bounded-lp}, given that both equality and inequality constraints are present, we ask for LICQ to maintain boundedness of $\pi_k$.
\begin{lemma}\label{lem:bounded-penalty}
	If LICQ of the constraints in~\eqref{eqn:opt-ms-simp} are satisfied at every accumulation point $\bar{x}$ of serious steps $\{x_k\}$ generated by Algorithm~\ref{alg:simp-bundle} and~\ref{alg:simp-bundle-const}, then 
	the penalty parameter sequence $\{\pi_k\}$ stabilizes, 
	This means there exists $k$ such that $\pi_t = \pi_k$ for all $t\geq k$.
\end{lemma}
\begin{proof}
     We rewrite the first equation in optimality condition in~\eqref{eqn:simp-bundle-penal-KKT-1} as
    \begin{equation} \label{eqn:simp-bundle-penal-KKT-full}
     \centering
     \begin{aligned}
	     g_k + \alpha_k d_k + \sum_{j=1}^m \lambda^{k+1}_j \nabla c_j(x_k) -\sum_{i=1}^n (\zeta_l^{k+1})_i e_i 
	     +\sum_{i=1}^n (\zeta_u^{k+1})_i e_i &=0,\\
     \end{aligned}
    \end{equation}
    where $e_i \in \Rbb^n$ is a vector such that $(e_i)_i = 1$ and $(e_i)_l = 0,l\neq i$. 
    The Lagrange multipliers for the bound constraints can be combined into one vector $\zeta^{k+1}=\zeta_l^{k+1}- \zeta_u^{k+1}$,
    since $(\zeta_l^{k+1})_i(\zeta_u^{k+1})_i = 0$.
	A component of $\zeta_l^{k+1}$ or $\zeta_u^{k+1}$ is unbounded if and only if the corresponding component in $\zeta^{k+1}$ is unbounded.
	Let $I$ be the index set of the active bound constraints, we have
    \begin{equation} \label{eqn:simp-bundle-penal-KKT-full-2}
     \centering
     \begin{aligned}
	      g_k + \alpha_k d_k = -\sum_{j=1}^m \lambda^{k+1}_j \nabla c_j(x_k) +\sum_{i \in I} (\zeta^{k+1})_i e_i.\\
     \end{aligned}
    \end{equation}
	Suppose contrary to the Lemma, $\{\pi_k\}$ is not bounded as $k\to\infty$. Let $\bar{x}$ be an accumulation point of $\{x_k\}$ and the subsequence $x_{k_u} \to \bar{x}$. We have $\pi_{k_u}\to \infty$.  
	Since $\{x_k\},\{g_k\}$ and $\{\alpha_k\}$ are bounded (Lemma~\ref{lem:sufficientdecrease-penal}), 
	the left-hand side of~\eqref{eqn:simp-bundle-penal-KKT-full-2} is bounded as $x_{k_u}\to\bar{x}$. 
	
	From~Lemma~\ref{lem:bounded-lp} and steps 5 in Algorithm~\ref{alg:simp-bundle-const}, an unbounded $\{\pi_k\}$ is caused by $\pi_k < \norm{\lambda^{k+1}}_{\infty} +\gamma_f$ occurring infinite number of times. 
	Therefore, the update rule of $\pi_k$ ensures that $\norm{\lambda^{k}}_{\infty} \to \infty$. In particular,  $\norm{\lambda^{k_u}}_{\infty} \to \infty$ while the corresponding subsequences $\{x_{k_u}\}$, $\{g_{k_u}\}$, $\{\alpha_{k_u}\}$ remain bounded.
	From LICQ at $\bar{x}$, $\nabla c_j(\bar{x}) \in \Rbb^n $ and $e_i,i \in I$ are linearly independent and bounded vectors. 
	 Regardless of the behavior of $\{\zeta^{k_u}\}$, the right-hand side of~\eqref{eqn:simp-bundle-penal-KKT-full-2} will be unbounded due to linear independence at $\bar{x}$ among the vectors.
	This is a contradiction. 
	Therefore, the sequence $\{\lambda_k\}$ is bounded, as well as $\{\pi_k\}$.
	Since $\pi_k$ is determined by $\lambda^k$, there exists $k$ such that $\pi_t=\pi_k$ for all $t\geq k$. 
\end{proof}

The global convergence analysis from Section~\ref{sec:alg-convg} stands as $\pi_k$  and $\theta_k$ are bounded. This is reflected in the following two theorems, which are similar to Theorem~\ref{thm:simp-KKT}.

%
\begin{theorem}\label{thm:simp-KKT-consistency-finite}
	Under Assumptions~\ref{assp:upperC2},~\ref{assp:boundedHc} and LICQ conditions of Lemma~\ref{lem:bounded-lp}, if Algorithm~\ref{alg:simp-bundle-const} is called finitely many times, every accumulation point of the serious point sequence $\{x_k\}$ generated from Algorithms~\ref{alg:simp-bundle}
	and~\ref{alg:simp-bundle-const} is a KKT point of~\eqref{eqn:opt-ms-simp}. 
\end{theorem}
Since there are only a finite number of consistency restoration steps, the linearized constraints become consistent for $k$ large enough. Thus, only Algorithm~\ref{alg:simp-bundle} is called for $k$ large enough. As a result, Theorem~\ref{thm:simp-KKT-consistency-finite} is a direct consequence of Theorem~\ref{thm:simp-KKT}.

\begin{theorem}\label{thm:simp-KKT-consistency-bounded}
	Under Assumptions~\ref{assp:upperC2},~\ref{assp:boundedHc} and LICQ conditions of Lemma~\ref{lem:bounded-lp} and~\ref{lem:bounded-penalty}, if Algorithm~\ref{alg:simp-bundle-const} is called infinitely many times,
	then any accumulation point of the serious point sequence $\{x_k\}$ generated by Algorithms~\ref{alg:simp-bundle} 
	and~\ref{alg:simp-bundle-const}
	is a KKT point of~\eqref{eqn:opt-ms-simp}.
\end{theorem}
\begin{proof}
The proof is similar to that of Theorem~\ref{thm:simp-KKT}.
	First, by Lemma~\ref{lem:sufficientdecrease} and~\ref{lem:sufficientdecrease-penal}, there exists $k_0$ such that $\alpha_t = \alpha_k$ for all $t\geq k_0$ and all steps produced by both algorithms are serious steps.
We note that by design both penalty parameters $\theta_k$ and $\pi_k$ in~\eqref{eqn:opt-ms-simp-bundle-merit} and~\eqref{def:merit-fea} are monotonically non-decreasing across iterations and algorithms. By Lemma~\ref{lem:bounded-lp}, $\lambda^k$ from Algorithm~\ref{alg:simp-bundle} is bounded. 
By Lemma~\ref{lem:bounded-penalty}, $\pi_k$ is bounded. From step 7 in Algorithm~\ref{alg:simp-bundle} and
step 5 in Algorithm~\ref{alg:simp-bundle-const}, there exists $k_1$ so that $\theta_t=\theta_k=\pi_t$ for all $t\geq k_1$.
Let $k>\max(k_0,k_1)$, all parameters for both algorithms stabilize and 
	the merit functions $\phi_{1\theta_k}(\cdot)$ of~\eqref{eqn:opt-ms-simp-bundle-merit} and  $\phi_{1\pi_k}(\cdot)$  of~\eqref{def:merit-fea} have the same form with $\theta_k=\pi_k:=\bar{\theta}=\bar{\pi}$
	, where $\bar{\theta}=\bar{\pi}$ is the stabilized value.

Since $\{x_k\}$ is bounded, there exists at least one accumulation point for it. Let $\bar{x}$ be an accumulation point and $\{x_{k_s}\}$ be the subsequence of $\{x_k\}$ such that $x_{k_s}\to\bar{x}$.
From Lemma~\ref{lem:line-search-merit-2}, line search terminates successfully in finite steps. 
	By Lemma~\ref{lem:sufficientdecrease-merit} and~\ref{lem:sufficientdecrease-merit-penal}, $\{\phi_{1\pi_k}(x_k)\}$ and $\{\phi_{1\theta_k}(x_k)\}$ decrease monotonically in the order of $\norm{d_k}^2$ and are bounded below. Therefore, they converge and $\{d_k\}$ generated by both algorithms satisfies $d_k \to 0$. In particular, $d_{k_s}\to 0$

Compared to Theorem~\ref{thm:simp-KKT-consistency-finite}, Theorem~\ref{thm:simp-KKT-consistency-bounded} leaves the possibility that after a finite number of calls of Algorithm~\ref{alg:simp-bundle}, only  Algorithm~\ref{alg:simp-bundle-const} is invoked.  
	From  step 5 in Algorithm~\ref{alg:simp-bundle-const} we have $\bar{\pi} > \lambda_j^{k_s+1} + \gamma_f$ for $j\in[1,m]$ and $k_s$ large enough. 
	Using the first two equalities in~\eqref{eqn:simp-bundle-penal-KKT-2}, we have $p_j^{k_s+1}\geq \gamma_f>0$. Passing on to a subsequence if necessary, we assume $p^{k_s} \to \bar{p}>0$, and thus by the third equation in~\eqref{eqn:simp-bundle-penal-KKT-2} $v^{k_s}\to \bar{v} = 0$. Similarly, we have $w^{k_s}\to 0$. From third equation in~\eqref{eqn:opt-ms-simp-bundle-penal}, we have $c(\bar{x}) = 0$.

Finally, passing to a subsequence if necessary, we let $g_{k_s}\to \bar{g}$, 
	$\lambda^{k_{s}+1} \to \bar{\lambda}$, $\zeta_u^{k_{s}+1} \to \bar{\zeta}_u$, $\zeta_l^{k_{s}+1} \to \bar{\zeta}_l$. 
        From the first equation in the optimality conditions~\eqref{eqn:simp-bundle-penal-KKT-1}, we have 
        \begin{equation} \label{eqn:simp-bundle-penal-KKT-limit-1}
		  0 = \bar{g} + \nabla c(\bar{x}) \bar{\lambda}  -\bar{\zeta}_l +\bar{\zeta}_u. 
        \end{equation}
	By the outer semicontinuity of Clarke subdifferential, with $g_{k_s}\in \bar{\partial} r(x_{k_s})$, we have $\bar{g} \in \bar{\partial} r(\bar{x})$.
        As a result, 
	  $0 \in \bar{\partial} r(\bar{x}) +\nabla c(\bar{x}) \bar{\lambda} -\bar{\zeta}_l +\bar{\zeta}_u$.
	The complementarity conditions of bound constraints from~\eqref{eqn:simp-bundle-penal-KKT-1} lead 
	to $\bar{Z}_u(\bar{x}-x_u)$, $\bar{Z}_l\bar{x}=0$. Together with the 
	equality constraints $c(\bar{x}) = 0$, the first-order necessary 
	optimality conditions~\eqref{eqn:opt-ms-simp-KKT} of problem~\eqref{eqn:opt-ms-simp} at $\bar{x}$ are satisfied.
	The proof above stands valid for the case where both algorithms are called infinitely many times as well.
\end{proof}

If LICQ is not satisfied at the accumulation points, the update rules for penalty parameters $\theta_k$ and $\pi_k$ can be modified to encourage convergence to critical points of the linearized constraint violation. It is also possible to ask for a less restrictive constraint qualification.
For example, in smooth optimization, under a less restrictive constraint qualification, Mangasarian–Fromovitz constraint qualification, Byrd et al.~\cite{byrd2005} show that if $\pi_k \to \infty$ and $c(\bar{x}) = 0$, then their algorithm converges 
to a KKT point. Such result for (nonsmooth) upper-$C^2$ objective function will be the subject of our future research.

\section{Numerical Applications}\label{sec:exp}
We present three numerical examples to demonstrate the theoretical and numerical 
capabilities of the proposed algorithm. They are chosen within 
the general formulation of two-stage optimization problems.
The first two problems are simple ones with analytic solutions and upper-$C^2$ objective that is not
lower-$C^2$ nor prox-regular. We show that the proposed algorithm provides improved convergence
behavior over algorithms that assume lower-type objectives, such as the redistributed bundle method of~\cite{hare2010}. 

\begin{example}\label{exmp:1}
	(Differentiable with unbounded gradient) Consider
\begin{equation}\label{prob:ex1-1st}
 \begin{aligned}
  &\underset{x }{\text{min}} 
	  &    f(x_1) + &\mu [(x_{2}-\frac{1}{2})^2+ x_3^2] +  r(x)\\
   &\text{s.t.}
	  &  -5 \leq &x_1 \leq 5, \ 0 \leq x_2 \leq 50,\ -1 \leq x_3 \leq 10,
 \end{aligned}
\end{equation}
where $x\in\Rbb^3$, $\mu=10^5$ and $f(x_1):\Rbb\to\Rbb$ is a continuously differentiable function.
	The objective $r:\Rbb^3\to\Rbb$ is the solution to the second-stage problem
\begin{equation}\label{prob:ex1-2nd}
 \begin{aligned}
	 &\underset{y\in\Rbb^3}{\text{min}}  &&  \norm{x-y}^2 \\
   &\text{s.t.}
	  &&  y_2 \leq y_3^2, \  -5 \leq y_1 \leq 5,\\
	  &&& -5 \leq y_2\leq 5, \ 0 \leq y_3 \leq 10.\\
 \end{aligned}
\end{equation}
\end{example}
The objective term $r(\cdot)$ is a squared distance function on a closed set and, thus, is upper-$C^2$~\cite{rockafellar1998}. 
In addition, it is lower-$C^1$, but not lower-$C^2$ at $\tilde{x}=[x_1,\frac{1}{2},0]$, $\forall x_1\in[-5,5]$. Thus, $r$ is continuously differentiable but its gradient is not bounded  
at $\tilde{x}$, as shown in the feasible region plot on the left of Figure~\ref{fig:ex_fs}.
In this case, the proposed algorithm offers global convergence support compared to algorithms
that require a lower-$C^2$ objective. We remark that $r(\cdot)$ is  
continuously differentiable at the other points and other algorithms with carefully 
chosen parameters can also succeed in solving Example 1.
\begin{figure}
  \centering
  \includegraphics[width=0.95\textwidth]{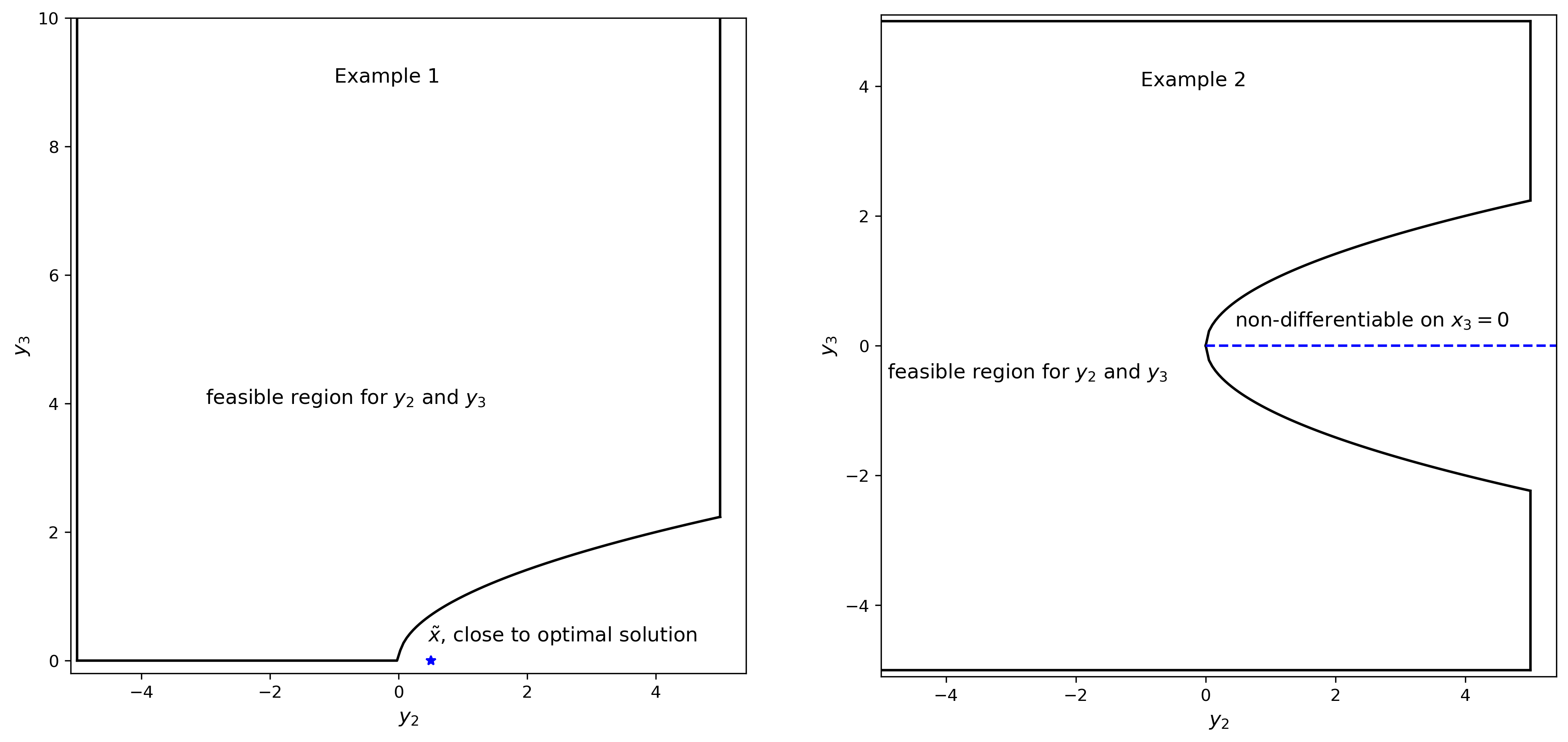}
	\caption{Feasible region of $y_2,y_3$ plane of example 1 (left) and example 2 (right)}
\label{fig:ex_fs}
\end{figure}

The true solution is obtained by treating the two-stage problem as one problem with variables in $\Rbb^6$ and being solved with Ipopt~\cite{ipopt}. In Algorithm~\ref{alg:simp-bundle}  $\eta_l^+$,$\eta_l^-$,$\eta_{\gamma}^+$,$\eta_{\gamma}^-$ are set to $1$, $\alpha_0=1.0,\epsilon=10^{-8}$, $\eta_{\beta}=0.5$, $\eta_{\alpha}=1.25$, and $\gamma=1$. The redistributed bundle method in~\cite{hare2010} is implemented as a comparison with $\Gamma=2,\mu_0=1,\eta_0=1$ and $\text{TOL}_{\text{stop}}=10^{-8}$.
The initial point is set to $x_0=[1,50,5]^T$. 
The simplified bundle method exits in 4 iterations.
While both algorithms quickly moved close to the solution, due to the lack of lower-$C^2$ property at $x=[x_1,\frac{1}{2},0]$ ($\forall x_1\in[-5,5]$), the convexification parameter $\eta_n$ registers a large value for the redistributed bundle method. 
The relatively small error tolerance requires more iterations  from the redistributed bundle method. 
Figure~\ref{fig:ex1} shows the result of error measure against the number of iterations for both redistributed (left) and simplified (right) bundle method. The quadratic coefficient, which decreases in the simplified bundle method as explained in Remark~\ref{rmrk:realalpha}, is also plotted for both algorithms.
\begin{figure}
  \centering
  \includegraphics[width=0.95\textwidth]{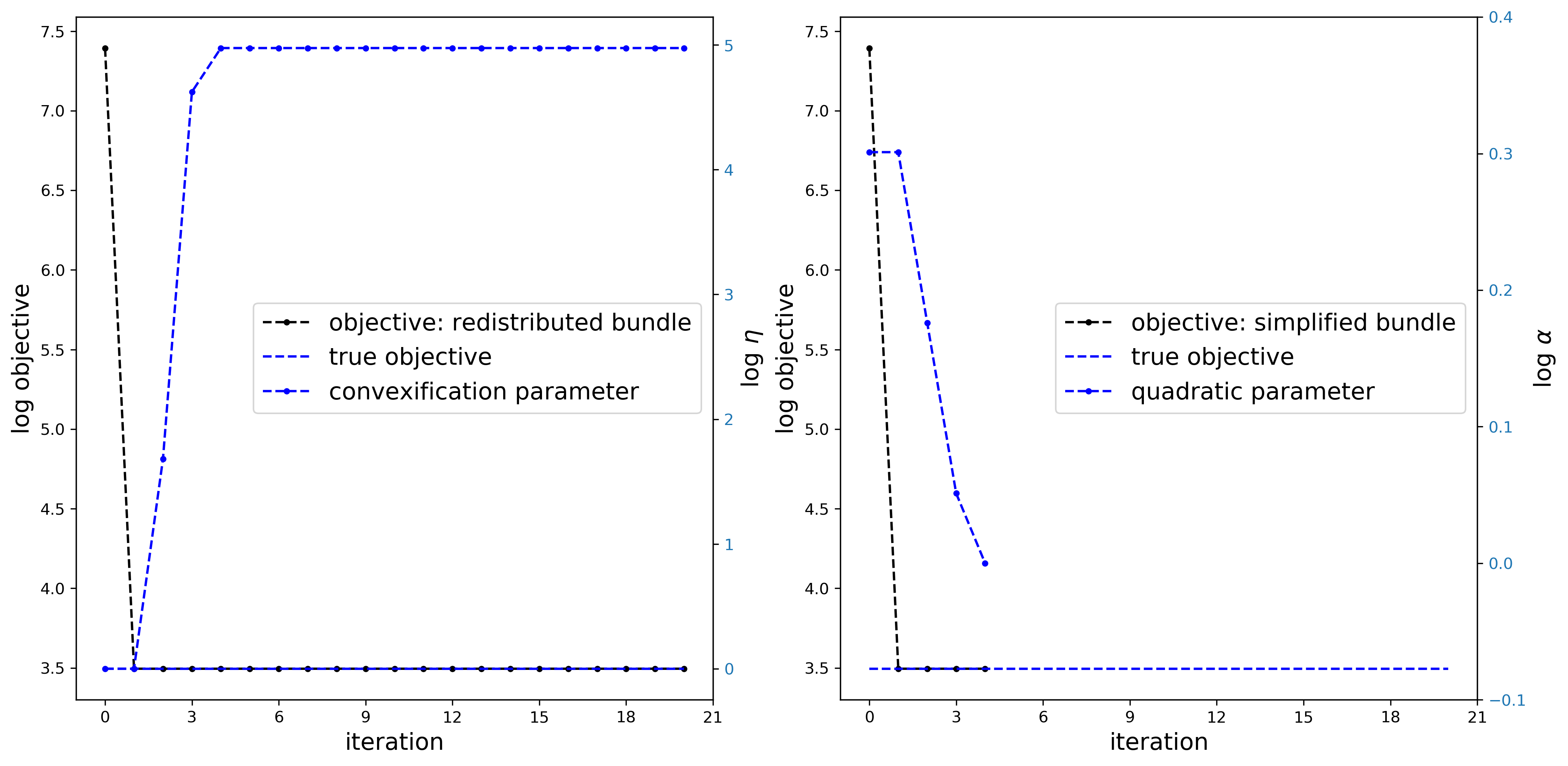}
 \caption{Convergence and quadratic coefficient plots for example 1}
\label{fig:ex1}
\end{figure}

\begin{example}\label{exmp:2}
  (Non-differentiable) Example 2 has the following form:
\begin{equation}\label{prob:ex2-1st}
 \begin{aligned}
  &\underset{x }{\text{min}} 
	  &    f(x_1) + &\mu [(x_{2}-\frac{1}{2})^2+ x_3^2] +  r(x)\\
   &\text{s.t.}
	  &  -5 \leq &x_1 \leq 5, \ \ 0 \leq x_2 \leq 50, \ \  -5 \leq x_3 \leq 5.\\
 \end{aligned}
\end{equation}
Again, $\mu=10^5$ and $f(x_1):\Rbb\to\Rbb$ is a continuously differentiable function.
	The function $r:\Rbb^3\to\Rbb$ is the solution to the second-stage problem
\begin{equation}\label{prob:ex2-2nd}
 \begin{aligned}
  &\underset{y \in \Rbb^3}{\text{min}} 
	  & & \norm{x-y}^2 \\
   &\text{s.t.}
	  &&  y_2 \leq y_3^2, \ -5 \leq y_1,y_2,y_3 \leq 5.\\
 \end{aligned}
\end{equation}
\end{example}
Example 2 is designed to vary slightly from Example 1 to illustrate the large group of problems 
the proposed algorithm can tackle. 
With a slight change in the constraint to allow $y_3<0$, the solution function $r$ is no longer differentiable on $x_3=0$ as 
multiple solutions $y$ exist. This is illustrated on the right plot in Figure~\ref{fig:ex_fs}.
However, $r$ remains upper-$C^2$ and the convergence analysis for the proposed algorithm applies.
Figure~\ref{fig:ex2} shows the objective and quadratic coefficient for both redistributed (left) and simplified (right) bundle method from the same starting point and with same parameters as in Example 1. Similar conclusions as in example 1 can be drawn.
\begin{figure}
  \centering
  \includegraphics[width=0.95\textwidth]{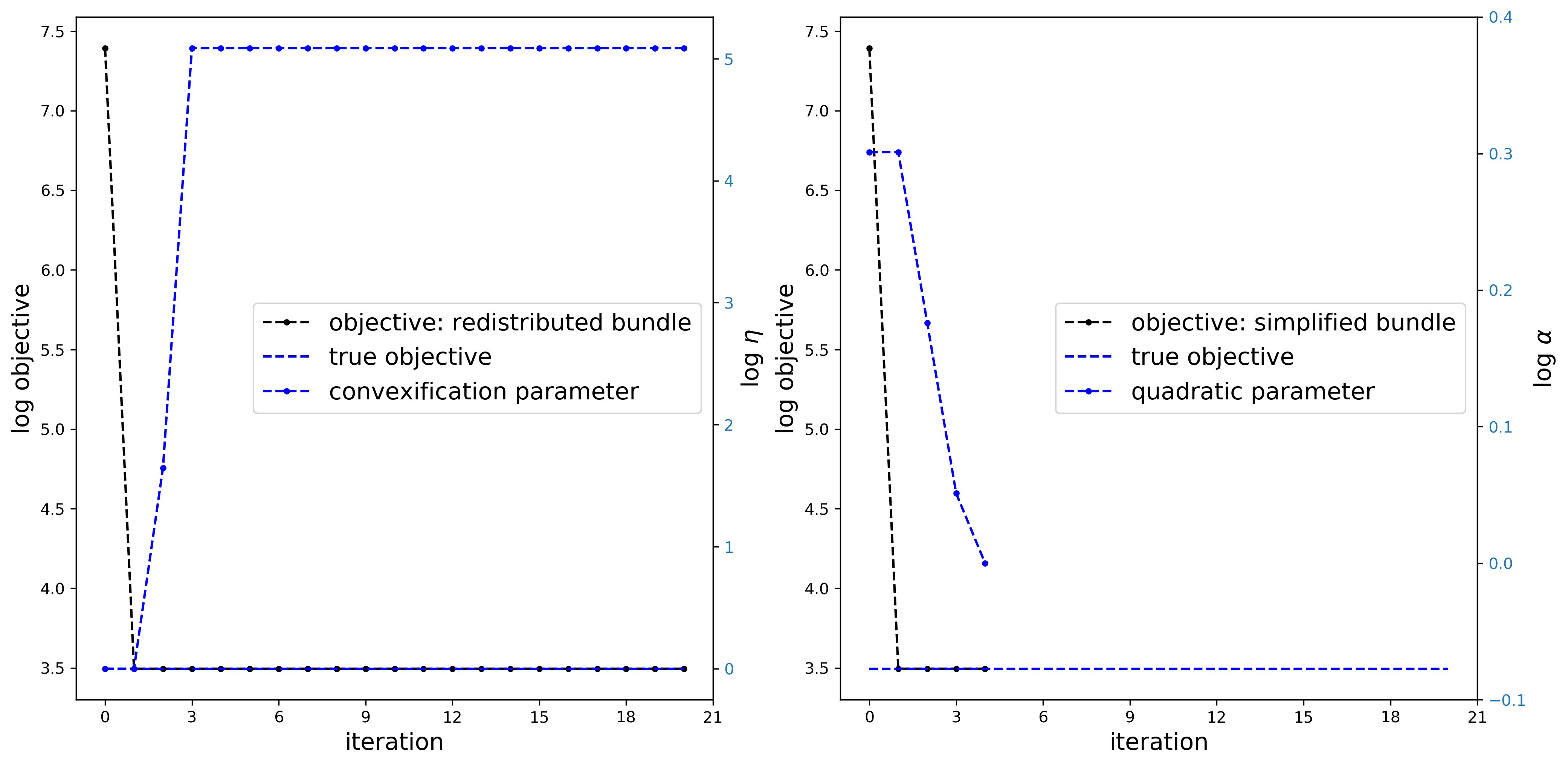}
 \caption{Convergence and quadratic coefficient plots for example 2}
\label{fig:ex2}
\end{figure}

\begin{example}\label{exmp:3}
	(smoothed SCACOPF) Example 3 is a SCACOPF problem with affine active power constraint for 
	contingency (second-stage) problems. The network data used in this example is from the ARPA-E Grid Optimization competition~\cite{petra_21_gollnlp}. 
	The complete mathematical formulation is complex but the master (first-stage) problem fits in the form of~\eqref{eqn:opt0}, 
	where $r$ is the recourse function of the contingency problems. 
	Details of the 
	problem setup can be seen in~\cite{petra_21_gollnlp}. The number of contingency problems that are solved to evaluate $r$ is 100.
\end{example}
The coupling constraint between master and contingency variables can be viewed as linear in the former ($x$) but it is nonsmooth. This means recourse $r(\cdot)$ might not be upper-$C^2$.
However, using a quadratic penalty of the coupling constraints in the contingency problems, $r$ in~\eqref{eqn:opt0} becomes upper-$C^2$ and the problem is referred to as the smoothed SCACOPF, in contrast to the original non-smoothed one. 
The proposed algorithm is applied to the smoothed SCACOPF, where the quadratic penalty parameter $\mu$ is set to $10^9$.
While this means the convergence analysis applies, we only solve an approximated problem.
To verify the accuracy of the solution, the true solution is obtained by solving the extensive form of the SCACOPF with Ipopt. 
It is plotted with the optimization results in Figure~\ref{fig:ex3}. We also plot the non-smoothed objective evaluated at the optimal solution $x$ gained from the smoothed problem at each iteration. 
The rejected steps are marked as well.
Within 200 iterations, the non-smoothed objective reach within $0.010\%$ error of the true solution, which
is acceptable and useful in practice. To speed up convergence of this first-order method, 
the quadratic coefficient $\alpha_k$ is reduced whenever possible. 
For large-scale problems with 
$10^4$ coupled optimization variables and $10^5$ contingencies, the extensive form of the SCACOPF
would be impractical, while the simplified bundle algorithm has proved to be scalable at large computational scale~\cite{wang2021}.

\begin{figure}
  \centering
  \includegraphics[width=0.95\textwidth]{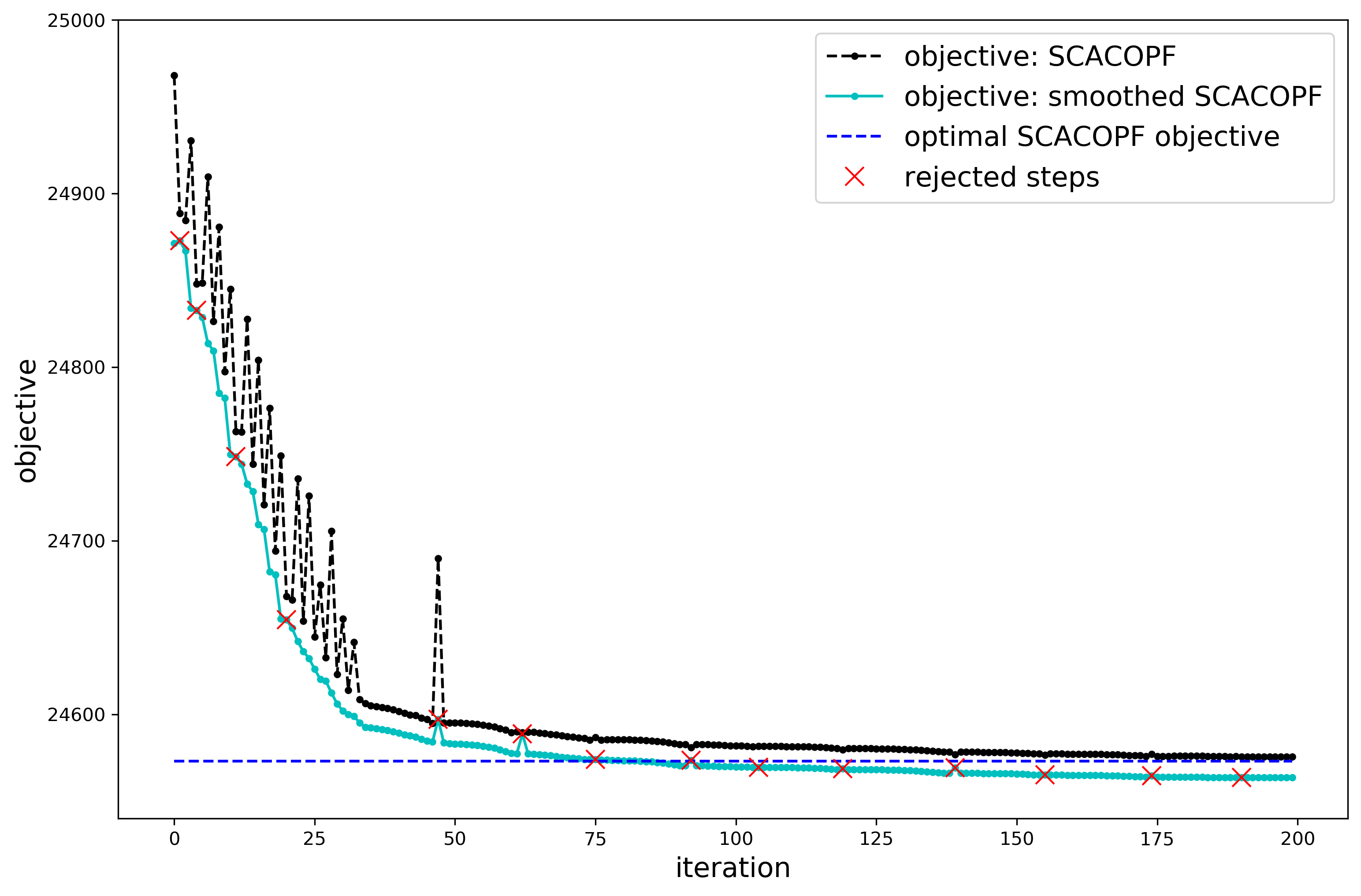}
 \caption{Convergence plots for example 3}
\label{fig:ex3}
\end{figure}

\section{\normalsize Conclusions}\label{sec:con}
In this paper, we have proposed and analyzed a two-algorithm methodology for nonsmooth, nonconvex optimization problems
with upper-$C^2$ objectives, which exist in many applications, particularly two-stage optimization problems.
While both bundle methods and DC algorithms can potentially solve problems with upper-$C^2$ objectives, attention is usually focused on the weakly convex objectives. 
The proposed algorithm fills the gap of a robust and efficient algorithm for problems that lack weak convexity.
To that end, a convex quadratic programming (QP) subproblem is solved at each iteration and a trust-region update 
rule is used to adjust the model parameters, which can stay bounded under reasonable assumptions. The line search 
on the constraints is carefully designed to ensure overall progress.
The penalty subproblem in the consistency restoration algorithm remains a QP problem and is effective in addressing the 
potential inconsistency of the linearized constraints.
Global convergence analysis of the proposed pair of algorithms is provided. 

The algorithms are presented from a bundle method point of view, 
but they can also can be viewed as an extension of SQP or a trust-region updated proximal variation of DCA with line search.
Further relaxation of the assumptions is possible, particularly for LICQ, which will be studied in the future. Another topic of interest is to extend the algorithms to more two-stage optimization problems using a combination of different bundle methods.
Finally, we note that the algorithm has been implemented on parallel computing platforms for our target application,  power-grid optimization problems, and has shown significant potential for computational scalability~\cite{wang2021}.

\appendix

\section*{Acknowledgments}
Prepared by LLNL under Contract DE-AC52-07NA27344.
This document was prepared as an account of work sponsored by an agency of the United States government. Neither the United States government nor Lawrence Livermore National Security, LLC, nor any of their employees makes any warranty, expressed or implied, or assumes any legal liability or responsibility for the accuracy, completeness, or usefulness of any information, apparatus, product, or process disclosed, or represents that its use would not infringe privately owned rights. Reference herein to any specific commercial product, process, or service by trade name, trademark, manufacturer, or otherwise does not necessarily constitute or imply its endorsement, recommendation, or favoring by the United States government or Lawrence Livermore National Security, LLC. The views and opinions of authors expressed herein do not necessarily state or reflect those of the United States government or Lawrence Livermore National Security, LLC, and shall not be used for advertising or product endorsement purposes.
\bibliographystyle{siamplain}
\bibliography{bibliography}

\begin{thebibliography}{10}

\bibitem{an2018}
{\sc L.~T.~H. An and P.~D. Tao}, {\em {DC programming and DCA: thirty years of
  developments}}, Mathematical Programming, 169 (2018), pp.~5--68.

\bibitem{apkarian2008}
{\sc P.~Apkarian, D.~Noll, and O.~Prot}, {\em A trust region spectral bundle
  method for nonconvex eigenvalue optimization}, SIAM J. Optim., 19 (2008),
  p.~281–306.

\bibitem{Birge97Book}
{\sc J.~R. Birge and F.~Louveaux}, {\em Introduction to Stochastic
  Programming}, Springer-Verlag, New York,, 1997.

\bibitem{bonnans_book}
{\sc J.~F. Bonnans and A.~Shapiro}, {\em Perturbation Analysis of Optimization
  Problems}, Springer, New York, 1~ed., 2000.

\bibitem{byrd2005}
{\sc R.~H. Byrd, N.~I. Gould, J.~Nocedal, and R.~A. Waltz}, {\em On the
  convergence of successive linear-quadratic programming algorithms}, SIAM J.
  Optim., 16 (2005), p.~471–89.

\bibitem{ChiangPetraZavala_14_PIPSNLP}
{\sc N.~{Chiang}, C.~G. {Petra}, and V.~M. {Zavala}}, {\em Structured nonconvex
  optimization of large-scale energy systems using pips-nlp}, in 2014 Power
  Systems Computation Conference, 2014, pp.~1--7.

\bibitem{clarke1983}
{\sc F.~Clarke}, {\em {Optimization and Nonsmooth Analysis}}, {John Wiley \&
  Sons New York}, 1983.

\bibitem{cui2021}
{\sc Y.~Cui and J.~S. Pang}, {\em Modern Nonconvex Nondifferentiable
  Optimization}, Society for Industrial and Applied Mathematics, 2021.

\bibitem{curtis2017}
{\sc F.~E. Curtis, T.~Mitchell, and M.~L. Overton}, {\em A {BFGS-SQP} method
  for nonsmooth, nonconvex, constrained optimization and its evaluation using
  relative minimization profiles}, Optimization Methods Software, 32 (2017),
  p.~148–181.

\bibitem{curtis2012}
{\sc F.~E. Curtis and M.~Overton}, {\em A sequential quadratic programming
  algorithm for nonconvex, nonsmooth constrained optimization}, SIAM J. Optim.,
  22 (2012), pp.~474--500.

\bibitem{daniilidis2004}
{\sc A.~Daniilidis and P.~Georgiev}, {\em Approximate convexity and
  submonotonicity}, Journal of Mathematical Analysis and Applications, 291
  (2004), p.~292–301.

\bibitem{dao2015}
{\sc M.~Dao}, {\em Bundle method for nonconvex nonsmooth constrained
  optimization}, Journal of Convex Analysis, 22 (2015), pp.~1061--1090.

\bibitem{dao2016}
{\sc M.~Dao, J.~Gwinner, D.~Noll, and N.~Ovcharova}, {\em Nonconvex bundle
  method with application to a delamination problem}, Computational
  Optimization and Applications, 65 (2016).

\bibitem{hare2015}
{\sc W.~Hare, C.~Sagastizabal, and M.~Solodov}, {\em A proximal bundle method
  for nonsmooth nonconvex functions with inexact information}, Computational
  Optimization and Applications, 63 (2015).

\bibitem{hare2010}
{\sc W.~Hare and C.~Sagastizábal}, {\em A redistributed proximal bundle method
  for nonconvex optimization}, SIAM J. Optim., 20 (2010), pp.~2442--73.

\bibitem{hong2018}
{\sc M.~Hong}, {\em A distributed, asynchronous, and incremental algorithm for
  nonconvex optimization: An {ADMM} approach}, IEEE Transactions on Control of
  Network Systems, 5 (2018), pp.~935--945.

\bibitem{KallWallace}
{\sc P.~Kall and S.~W. Wallace}, {\em Stochastic Programming}, John Wiley \&
  Sons, Chichester, 2nd~ed., 1994.

\bibitem{kiwiel1985}
{\sc K.~Kiwiel}, {\em A linearization algorithm for nonsmooth minimization},
  Mathematics of Operations Research, 10 (1985), pp.~185--94.

\bibitem{Kiwiel1996}
{\sc K.~Kiwiel}, {\em Restricted step and levenberg-marquardt techniques in
  proximal bundle methods for nonconvex nondifferentiable optimization}, SIAM
  J. Optim., 6 (1996), pp.~227--249.

\bibitem{lemarechal1978}
{\sc C.~Lemaréchal}, {\em Bundle methods in nonsmooth optimization}, in
  Nonsmooth optimization (Proc. IIASA Workshop, Laxenburg, 1977), vol.~3,
  Pergamon, Oxford, 1978, pp.~79--102.

\bibitem{lemarechal2001}
{\sc C.~Lemaréchal}, {\em Lagrangian relaxation}, in Computational
  Combinatorial Optimization: Optimal or Provably Near-Optimal Solutions,
  M.~Jünger and D.~Naddef, eds., Springer Berlin Heidelberg, Berlin,
  Heidelberg, 2001, p.~112–156.

\bibitem{lemarechal1996}
{\sc C.~Lemaréchal and C.~Sagastizabal}, {\em Variable metric bundle methods:
  From conceptual to implementable forms}, Math. Program., 76 (1996),
  pp.~393--410.

\bibitem{liu2020}
{\sc J.~Liu, Y.~Cui, J.~S. Pang, and S.~Sen}, {\em Two-stage stochastic
  programming with linearly bi-parameterized quadratic recourse}, SIAM J.
  Optim., 30 (2020), p.~2530–2558.

\bibitem{mifflin1982}
{\sc R.~Mifflin}, {\em A modification and an extension of {Lemarechal’s}
  algorithm for nonsmooth minimization}, in Nondifferential and Variational
  Techniques in Optimization, vol.~17 of Mathematical Programming Studies,
  Springer, Berlin, Heidelberg, 1982, pp.~77--90.

\bibitem{mordukhovich2004upp}
{\sc B.~Mordukhovich}, {\em Necessary conditions in nonsmooth minimization via
  lower and upper subgradients}, Set-Valued Analysis, 12 (2004), pp.~163--193.

\bibitem{makela1992}
{\sc M.~M. Mäkelä and P.~Neittaanmäki}, {\em Nonsmooth Optimization}, WORLD
  SCIENTIFIC, 1992.

\bibitem{Nocedal_book}
{\sc J.~Nocedal and S.~J. Wright}, {\em Numerical Optimization}, Springer, New
  York, 2nd~ed., 2006.

\bibitem{noll2009}
{\sc D.~Noll}, {\em Cutting plane oracles to minimize non-smooth non-convex
  functions}, Set-Valued and Variational Analysis, 18 (2009), pp.~531--568.

\bibitem{noll2013}
{\sc D.~Noll}, {\em Bundle method for non-convex minimization with inexact
  subgradients and function values}, Springer Proceedings in Mathematics and
  Statistics, 50 (2013).

\bibitem{petra_21_gollnlp}
{\sc C.~G. Petra and I.~Aravena}, {\em Solving realistic security-constrained
  optimal power flow problems}, Operations Research, submitted (2021).

\bibitem{petra_14_realtime}
{\sc C.~G. Petra, O.~Schenk, and M.~Anitescu}, {\em Real-time stochastic
  optimization of complex energy systems on high performance computers},
  Computing in Science and Engineering, 99 (2014), pp.~1--9.

\bibitem{petra_14_augIncomplete}
{\sc C.~G. Petra, O.~Schenk, M.~Lubin, and K.~G{\"a}rtner}, {\em An augmented
  incomplete factorization approach for computing the {S}chur complement in
  stochastic optimization}, SIAM Journal on Scientific Computing, 36 (2014),
  pp.~C139--C162.

\bibitem{Qiu2005}
{\sc W.~Qiu, A.~J. {Flueck}, and F.~Tu}, {\em A parallel algorithm for security
  constrained optimal power flow with an interior point method}, in IEEE Power
  Engineering Society General Meeting, 2005, 2005, pp.~447--453 Vol. 1.

\bibitem{rockafellar1998}
{\sc R.~T. Rockafellar and R.~J.-B. Wets}, {\em Variational Analysis},
  Springer-Verlag, Berlin Heidelberg, 1998.

\bibitem{schramm1992}
{\sc H.~Schramm and J.~Zowe}, {\em A version of the bundle idea for minimizing
  a nonsmooth function: Conceptual idea, convergence analysis, numerical
  results}, SIAM J. Optim., 2 (1992), pp.~121--152.

\bibitem{Shapiro_book}
{\sc A.~Shapiro, D.~Dentcheva, and A.~Ruszczyński}, {\em Lectures on
  Stochastic Programming: Modeling and Theory, Second Edition}, Society for
  Industrial and Applied Mathematics, Philadelphia, PA, 2014.

\bibitem{shor1985}
{\sc N.~Z. Shor}, {\em Minimization methods for non-differentiable functions},
  Springer-Verlag, Berlin Heidelberg, 3~ed., 1985.

\bibitem{Spingarn1981SubmonotoneSO}
{\sc J.~Spingarn}, {\em Submonotone subdifferentials of lipschitz functions},
  Transactions of the American Mathematical Society, 264 (1981), pp.~77--89.

\bibitem{wang2021}
{\sc J.~Wang, N.~Y. Chiang, and C.~G. Petra}, {\em An asynchronous
  distributed-memory optimization solver for two-stage stochastic programming
  problems}, in 20th International Symposium on Parallel and Distributed
  Computing (ISPDC), IEEE, Jul 2021, pp.~33--40.

\bibitem{wang2022}
{\sc J.~Wang and C.~G. Petra}, {\em A simplified nonsmooth nonconvex bundle
  method with applications to security-constrained acopf problems}, 2022,
  \url{https://arxiv.org/abs/arXiv:2203.17215}.

\bibitem{wang2019}
{\sc Y.~Wang, W.~Yin, and J.~Zeng}, {\em Global convergence of admm in
  nonconvex nonsmooth optimization}, J. Sci. Comput., 78 (2019), p.~29–63.

\bibitem{ipopt}
{\sc A.~Wächter and L.~Biegler}, {\em On the implementation of an
  interior-point filter line-search algorithm for large-scale nonlinear
  programming}, Math. Program., 106 (2006), pp.~25--27.

\bibitem{xu2015}
{\sc M.~Xu, J.~Ye, and L.~Zhang}, {\em Smoothing {SQP} methods for solving
  degenerate nonsmooth constrained optimization problems with applications to
  bilevel programs}, SIAM J. Optim., 25 (2015), pp.~1388--1410.

\bibitem{yang2014}
{\sc Y.~Yang, L.~Pang, X.~Ma, and J.~Shen}, {\em Constrained nonconvex
  nonsmooth optimization via proximal bundle method}, J. Optim. Theory Appl.,
  163 (2014), p.~900–925.

\bibitem{zowe1989}
{\sc J.~Zowe}, {\em {The BT-Algorithm for Minimizing a Nonsmooth Functional
  Subject to Linear Constraints}}, in Nonsmooth Optimization and Related
  Topics, vol.~43 of Ettore Majorana International Science Series, Springer,
  Boston, MA, 1989.

\end{thebibliography}
\end{document}